\documentclass{article}
\usepackage{graphicx} 
\usepackage{lipsum}
\usepackage{lscape}
\usepackage{amsthm}
\usepackage{MnSymbol}
\newtheorem{theorem}{Theorem}[section]
\newtheorem{corollary}[theorem]{Corollary}
\usepackage{pst-all}
\usepackage{pstricks}    
\usepackage{graphicx} 
\usepackage{xfp}
\usepackage[utf8]{inputenc}
\usepackage{tikz}
\usepackage{rotating}
\usetikzlibrary{shapes.geometric}
\usepackage{amsmath}
\usepackage{pifont}
\usepackage{amsfonts}
\usepackage{ragged2e}
\usepackage{calculator}
 \usepackage{calculus}

\newtheorem{example}[theorem]{Example}
\newtheorem{prop}[theorem]{Proposition}
\theoremstyle{definition}
\theoremstyle{remark}
\newtheorem{remark}[theorem]{Remark}
\newtheorem{lemma}[theorem]{Lemma}

\newtheorem{definition}[theorem]{Definition}
\usepackage{tkz-euclide}
\usepackage{pgfplots}

\usepgflibrary{shapes.geometric} 
\usepgflibrary[shapes.geometric] 
\usetikzlibrary{shapes.geometric} 
\usetikzlibrary[shapes.geometric] 
\usetikzlibrary{shapes,decorations} 
\usepackage{graphicx} 
\renewcommand{\abstractname}{Abstract}
\NewDocumentCommand{\showcalculation}{o m}{$
  \IfValueTF{#1}
    {#1}{#2} = \fpeval{#2}
  $}

\renewenvironment{abstract}
 {\par\noindent\textbf{\abstractname.}\ \ignorespaces}
 {\par\medskip}
\begin{document}
\title{Heronian friezes and Plücker relations}
\author{Anja \v{S}neperger}
\date{}
\newcommand{\Addresses}{{
  \bigskip
  \footnotesize
   \textsc{School of Mathematics \\ University of Leeds \\ Leeds, LS2 9JT \\ United Kingdom\\}\par\nopagebreak \noindent
  \textit{E-mail address}: \texttt{mmasn@leeds.ac.uk}}}
  \maketitle
\begin{abstract} \noindent
In this article, we use Plücker relations in the Grassmannian $Gr(3,n)$ to give relations that hold amongst some of the entries of the Heronian frieze of order $n$. Furthermore, we make a connection between certain subfriezes of a Heronian frieze  and Plücker friezes $P(3,n)$, and then show that some  determinants of the matrices whose elements lie in those subfriezes are vanishing.
\end{abstract}
\tableofcontents
\section*{Introduction} Frieze patterns were introduced by Coxeter \cite{key7}, and classified by Conway and Coxeter \cite{key8} in terms of triangulations of polygons. In 2019, Fomin and Setiabrata introduced \textit{Heronian friezes} \cite{fs}. As well as the theory of cluster algebras of type $A$, their main motivation was the fact that measurements of a triangulation of a plane quadrilateral satisfy some equations from classical geometry. One of these is  Heron's formula for the area of a triangle, which gives friezes their name. Amongst many other interesting results, they prove that a general Heronian frieze possesses glide symmetry and establish a version of the Laurent phenomenon for these friezes.  \\
In this paper, we consider Heronian friezes arising from a general polygon in the plane (as defined in \cite{fs}), without a restriction on it being cyclic (as opposed to \cite{key2}, where this condition was assumed). We introduce the notion of a \textit{Heronian minor}, which stands for a minor of a certain matrix, of a form corresponding to the form of the entries of the frieze that represent areas of triangles (Definition \ref{her_minor}). We also introduce the notion of a \textit{Heronian minor relation}, which represents a Plücker relation for the Grassmannian $Gr(3,n)$ containing only Heronian minors (Definition \ref{her_mnr_rln}).  As in \cite{key2}, we give some relations between the entries of the frieze corresponding to areas of triangles, but here we use Plücker relations in the Grassmannian $Gr(3,n)$ \cite{key5}. In Theorem \ref{main}, we give an equivalent condition for a Plücker relation in $Gr(3,n)$ to be a Heronian minor relation. By making a connection to the Plücker friezes for $k=3$ \cite{key6}, in Theorem \ref{bitna} we show the vanishing of certain determinants of matrices whose elements lie in the subfrieze of entries corresponding to triangle areas. The paper is organised as follows.\\ In Section 1, we give main definitions and results concerning Heronian friezes, as stated in \cite{fs}, as well as recall some remarks from \cite{key2}. In Sections 2 and 3, we give an overview of the main definitions and theorems regarding exterior algebra and the Grassmannian $Gr(k,n)$, as stated in \cite{key5}. In Section 4, we introduce the notions of \textit{Heronian minor} and  \textit{Heronian minor relation}. Then we state and prove Theorem \ref{main}, as well as give an example (Example \ref{primjer}) where we list all the Heronian minor relations of the polygonal  Heronian frieze corresponding to a $6$-gon. Following on that, we recall some results on Plücker friezes (as stated in \cite{key6}), and then use them to establish and prove Theorem \ref{bitna}.
\section{Heronian friezes}
\begin{definition}\cite[Definition 2.3]{fs}\label{def:def100}
A \textit{Heronian diamond} is an ordered $10$-tuple of complex numbers $(a,b,c,d,e,f,p,q,r,s)$ satisfying the following equations:\begin{equation} p^2=H(b,c,e) \end{equation}
\begin{equation} q^2=H(a,d,e)\end{equation}
\begin{equation} r^2=H(a,f,b)\end{equation}
\begin{equation} s^2=H(c,f,d)\end{equation}
\begin{equation} r+s=p+q \end{equation}
\begin{equation} 4ef=(p+q)^2+(a-b+c-d)^2\end{equation}
\begin{equation} e(r-s)=p(a-d)+q(b-c),\end{equation} where $$
 H(x,y,z):= -x^2-y^2-z^2+2xy+2xz+2yz.$$ Instead of listing the components of a Heronian diamond as a row of $10$ numbers, we will typically arrange them in a diamond pattern as in \cite{fs} and as shown in Figure \ref{fig1}.
 \begin{figure}
 \centering
 \includegraphics[scale=1] {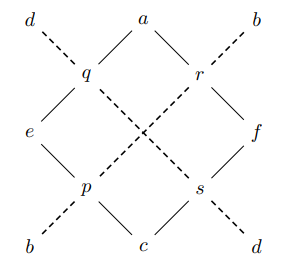}
 \caption {A Heronian diamond. Here, $b$ and $d$ are associated to the dashed lines extending the bimedians of the diamond. The remaining eight numbers are placed at the vertices of the diamond and at the midpoints of its sides. }
\label{fig1}
 \end{figure}
 \label{def}
 \end{definition}
 Motivated by Definition \ref{def}, in \cite{fs} the authors introduced the definition of a \textit{Heronian frieze}: roughly speaking, a Heronian frieze is a collection of complex numbers arranged  as in Figure \ref{figura102}, and satisfying the Heronian diamond equations for all the diamonds in the pattern (plus some boundary conditions). 
\begin{figure}
\centering
\includegraphics[scale=0.2]{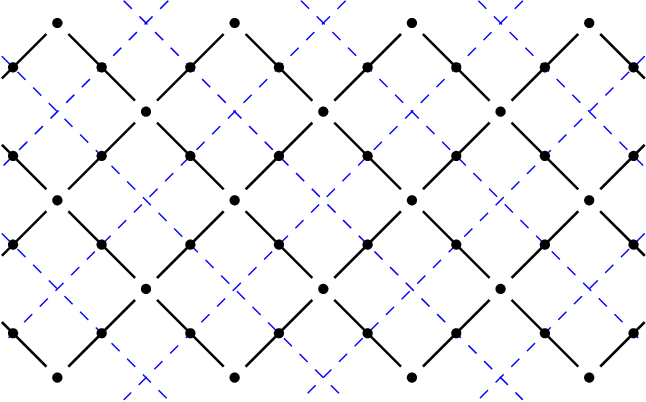}
\caption{The combinatorial pattern underlying a Heronian frieze of order $n=4$}
\label{figura102}
\end{figure}
Before passing on to a formal definition of Heronian frieze, let us first introduce a definition of a  labeled polygon, and some more motivation behind the actual notion of a Heronian diamond.
\begin{definition} Let $n \geq 3$ be an integer. Then a labeled \textit{polygon} (specifically an $n$-gon) in the complex plane $A$ is an ordered $n$-tuple of vertices $P = (A_1, A_2, ..., A_n) \in A^n$.
    \end{definition}
    \noindent
 Let $(A_1, A_2, ..., A_n)$ be a  labeled polygon in the complex plane. Then let $x_{ij}$ denote the squared distance between points $A_i$ and $A_j$, and $S_{ijk}$ denote four times the signed area of the triangle $A_iA_jA_k$, where $i,j, k$ $\in\{1,2,...,n\}$. \\ Namely, if $A_m = (x_m,y_m)$, for $m \in \{i,j,k\}$, then (as defined in \cite{fs})
\begin{equation}
x_{ij}: = (x_j - x_i)^2 + (y_j - y_i)^2, \label{eqx}
\end{equation} 
and 
\begin{equation}S_{ijk} := 2[(x_j - x_i)(y_k - y_i) - (y_j - y_i)(x_k - x_i)]\label{eqs}.
\end{equation}
[If  $i \leq 0$, $j \leq 0$, or $k \leq 0$,  we adopt a convention that $S_{ijk} = x_{ij} = 1$.]
\begin{remark}
When the vertices $A_i, A_j, A_k$ are ordered anticlockwise, $S_{ijk}$ is positive, i.e. equal to four times the actual area of the triangle. Moreover, the following equalities hold in addition \cite{fs}: 
\begin{equation}
S_{ijk} = - S_{ikj} = - S_{jik} = S_{jki} = S_{kij} = - S_{kji}
\end{equation}
\label{rem13}
\end{remark}
As explained in \cite{fs}, the motivation for the definition of a Heronian diamond actually came from triangulating a quadrilateral. Namely, a quadrilateral $(A_1, A_2, A_3, A_4)$ has two triangulations, involving diagonals $A_1A_3$ and $A_2A_4$, respectively. 
Figure \ref{fig2} shows these two triangulations, along with their respective measurement data:
\begin{equation} a = x_{14}, b = x_{12}, c = x_{23}, d = x_{34}, e = x_{13}, f = x_{24}, \end{equation}
\begin{equation} p = S_{123}, q = S_{134}, r = S_{124}, s = S_{234}.\end{equation}
\begin{figure}
\centering
\includegraphics[scale=1]{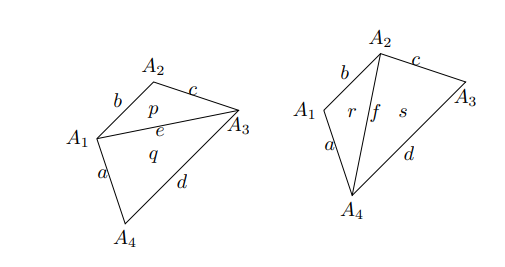}
\caption{Two triangulations of a plane quadrilateral}
\label{fig2}
\end{figure} 
Fomin and Setiabrata have shown \cite[Prop. 2.8]{fs} that the measurements (11) and (12) satisfy the equations (1) -- (7). Hence, the measurements related to the two triangulations of a $4$-gon give rise to a Heronian diamond, and represent a geometrical motivation for introducing this notion in general. Using this fact and previously introduced notation for squared distances and four times signed areas, a Heronian diamond for a quadruple of vertices $A_i, A_j, A_k, A_l$ can be represented as in Figure \ref{fig3}. 
\begin{remark} \label{labela}
As indicated in Figure \ref{fig3}, we will be labelling a Heronian diamond corresponding to vertices $A_i,A_j,A_k,A_l$ by $ijkl$.
\end{remark}
 Motivated by Definition ~\ref{def:def100}, Fomin and Setiabrata also introduce a notion of a \textit{Heronian frieze}. Here, we give the simplified version of the definition from \cite{fs} (as restated in \cite{key2}).
\begin{definition}
A \textit{Heronian frieze of order $n$} is a collection of complex numbers $z_{ij}, \tilde{z}_{ijk}$, for $i,j,k \in \{1,2,...,n\}$, arranged in a pattern shown in Figure \ref{friz}, such that every diamond of the pattern is Heronian, in the sense of Definition \ref{def:def100}. In addition, we impose the boundary conditions:
\begin{equation}
    z_{ii} = \tilde{z}_{iji} = \tilde{z}_{iij} = \tilde{z}_{ijj} = 0,
    \label{bound}
\end{equation}
for $i,j$ distinct elements of $\{1,2,...,n\}$.
\label{frizz}
 \end{definition}
 If all the entries associated to the dashed lines of the Heronian frieze of order $n$ are equal, i.e. $z_{12} = z_{23} = z_{34} = ... = z_{(n-1)n} = z_{n1}$, we say that the frieze is \textit{equilateral} \cite{fs}.
\begin{remark}Note that every entry of a Heronian frieze of order $n$ is in one of the forms $z_{ij}$, $\tilde{z}_{i,i+1,j}$, or $\tilde{z}_{i,j,j+1}$, where $i, j \in \{1,2,...,n\}$, and addition is modulo $n$.
\label{napomena}
\end{remark}
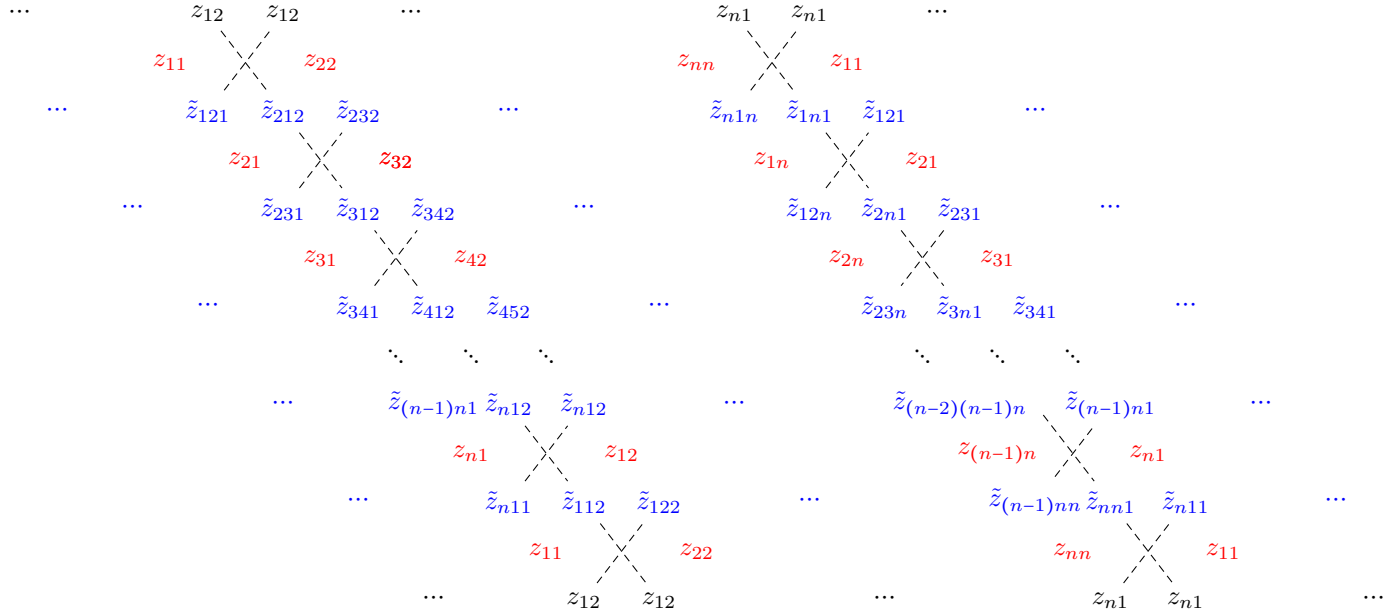
\begin{figure}
\begin{center}
\hspace*{-1.6cm}
\begin{tikzpicture}[xscale=1,yscale=1.3]
  \node[red,very thick](n1) at (4,4) {$z_{11}$};
   \node[blue,very thick](n2) at (4.5,3.5) {$\tilde{z}_{121}$};
   \node[blue,very thick](n15951951) at (2.5,3.5){$\cdots$};
   \node[draw=none] (n51621712) at (5.5,4.5) {$z_{12}$};
    \draw[densely dashed] (n2) -- (n51621712);
   \node[draw=none] (n51621712) at (4.5,4.5) {$z_{12}$};
\node[draw=none](n598515914151) at (2,4.5){$\cdots$};
   \node[draw=none] (n5262722) at (7.2,4.5) {$\cdots$};
    \node[red,very thick](n3) at (5,3) {$z_{21}$};
    \node[blue,very thick](n5) at (5.5,2.5) {$\tilde{z}_{231}$};
\node[blue,very thick](n591991) at (3.5,2.5){$\cdots$};
    \node[red,very thick](n4) at (6,2) {$z_{31}$};
    \node[blue,very thick](n55) at (6.5,1.5) {$\tilde{z}_{341}$};
    \node[blue,very thick](n8581581) at (4.5,1.5){$\cdots$};
    \node[red,very thick](n9) at (8,0) {$z_{n1}$};
    \node[draw=none](n12) at (7,1) {$\ddots$};
    
    \node[blue,very thick](n10) at (7.5,0.5) {$\tilde{z}_{(n-1)n1}$};
    \node[blue,very thick](n591561) at (5.5,0.5){$\cdots$};
    \node[blue,very thick](n7) at (8.5,-0.5) {$\tilde{z}_{n11}$};
   \node[blue,very thick](n419419) at (6.5,-0.5){$\cdots$};
  
    \node[red,very thick](n8) at (9,-1) {$z_{11}$};
    \node[draw=none] (n14151) at (9.5,-1.5) {$z_{12}$};
   \node[draw=none](n5188511) at (7.5,-1.5){$\cdots$};
    \node[red,very thick](n13) at (6,4) {$z_{22}$};
    
    \node[blue,very thick](n8043879) at (8.5,3.5) {$\cdots$};
    \node[blue,very thick](n17) at (6.5,3.5) {$\tilde{z}_{232}$};
\draw[densely dashed](n5) -- (n17);
    \node[red,very thick](n14) at (7,3) {$z_{32}$};
    \node[blue,very thick](n87) at (7.5,2.5) {$\tilde{z}_{342}$};
    \draw[densely dashed](n55)--(n87);
    \node[blue,very thick](n956116) at (9.5,2.5){$\cdots$};
  \node[blue,very thick](n951991) at (12.5,2.5){$\tilde{z}_{12n}$};
  \node[blue,very thick](n1161996119) at (12.5,3.5){$\tilde{z}_{1n1}$};
  \node[blue,very thick](n599519961) at (11.5,3.5){$\tilde{z}_{n1n}$};
  \node[draw=none](n4194190) at (12.5,4.5){$z_{n1}$};
  \draw[densely dashed](n599519961)--(n4194190);
  \node[draw=none](n95159159) at (14.2,4.5){$\cdots$};
  \node[draw=none](n419151) at (11.5,4.5){$z_{n1}$};
  \draw[densely dashed](n419151)--(n1161996119);
  \node[red,very thick](n519519151) at (11,4){$z_{nn}$};
  \node[blue,very thick](n4949494491149) at (13.5,3.5){$\tilde{z}_{121}$};
  \draw[densely dashed](n951991) -- (n4949494491149);
  \node[blue,very thick](n599519859) at (15.5,3.5){$\cdots$};
  \node[red,very thick](n95159119) at (12,3){$z_{1n}$};
  \node[red,very thick](n5959595) at (13,2){$z_{2n}$};
  \node[red,very thick](n919599) at (13,4){$z_{11}$};
  \node[red,very thick](n4194149) at (14,3){$z_{21}$};
  \node[blue,very thick](n151951519) at (14.5,2.5){$\tilde{z}_{231}$};
  \node[blue,very thick](n915851591) at (16.5,2.5){$\cdots$};
  \node[red,very thick](n51519951) at (15,2){$z_{31}$};
  \node[blue,very thick](n595155151) at (15.5,1.5){$\tilde{z}_{341}$};
  \node[blue,very thick](n915115) at (17.5,1.5){$\cdots$};
  \node[draw=none](n591551166) at (16,1){$\ddots$};
  \node[blue,very thick](n51519518) at (16.5,0.5){$\tilde{z}_{(n-1)n1}$};
  \node[blue,very thick](n4185851) at (18.5,0.5){$\cdots$};
  \node[red,very thick](n591992592) at (17,0){$z_{n1}$};
  \node[blue,very thick](n949195) at (13.5,2.5){$\tilde{z}_{2n1}$};
  \draw[densely dashed](n1161996119) -- (n949195);
  \node[blue,very thick](n95151591) at (14.5,1.5){$\tilde{z}_{3n1}$};
  \node[draw=none](n95519196) at (15,1){$\ddots$};
  \draw[densely dashed](n949195) -- (n95151591);
  \node[blue,very thick](n95151541) at (13.5,1.5){$\tilde{z}_{23n}$};
  \draw[densely dashed](n151951519) -- (n95151541);
  \node[draw=none](n916916916) at (14,1){$\ddots$};
\node[blue,very thick](n915195141) at (14.5,0.5){$\tiny{\tilde{z}_{(n-2)(n-1)n}}$};
\node[blue,very thick](n59111581) at (15.5,0.5){};
\node[red,very thick](n599150) at (15,0){$z_{(n-1)n}$};
\node[blue,very thick](n95519581) at (15.5,-0.5){$\tilde{z}_{(n-1)nn}$};
\draw[densely dashed](n51519518) -- (n95519581);
\node[red,very thick](n699234113) at (16,-1){$z_{nn}$};
\node[draw=none](n91515151) at (16.5,-1.5){$z_{n1}$};
\node[blue,very thick](n59191591559) at (16.5,-0.5){ $ \tilde{z}_{nn1}$};
\draw[densely dashed](n59111581)--(n59191591559);
\node[blue,very thick](n95916661511) at (17.5,-0.5){$\tilde{z}_{n11}$};
\node[blue,very thick](n491951941) at (19.5,-0.5){$\cdots$};
\draw[densely dashed](n91515151)--(n95916661511);
\node[red,very thick](n544949449494) at (18,-1){$z_{11}$};
\node[draw=none](n499151161) at (17.5,-1.5){$z_{n1}$};
\node[draw=none](n519951915) at (20,-1.5){$\cdots$};
\draw[densely dashed](n59191591559) -- (n499151161);
    \node[red,very thick](n14) at (8,2) {$z_{42}$};
    \node[blue,very thick](n14) at (8.5,1.5) {$\tilde{z}_{452}$};
    \node[blue,very thick](n195551161) at (10.5,1.5){$\cdots$};
    \node[draw=none](n22) at (9,1) {$\ddots$};
    \node[blue,very thick](n24) at (9.5,0.5) {$\tilde{z}_{n12}$};
    \node[blue,very thick](n951815) at (11.5,0.5){$\cdots$};
    \draw[densely dashed] (n7) -- (n24); 
    \node[red,very thick](n65) at (10,0) {$z_{12}$};
    \node[blue,very thick](n32) at (10.5,-0.5) {$\tilde{z}_{122}$};
    \node[blue,very thick](n1816619) at (12.5,-0.5){$\cdots$};
     \draw[densely dashed] (n14151) -- (n32);
    \node[red,very thick](n90) at (11,-1) {$z_{22}$};
    \node[blue,very thick](n90) at (5.5,3.5) {$\tilde{z}_{212}$};
       \draw[densely dashed] (n51621712) -- (n90);
    \node[blue,very thick](n19) at (6.5,2.5) {$\tilde{z}_{312}$};
    \draw[densely dashed](n90)--(n19);
    \node[blue,very thick](n14) at (7.5,1.5) {$\tilde{z}_{412}$};
    \draw[densely dashed](n14)--(n19);
    \node[draw=none](n14) at (8,1) {$\ddots$};
    \node[blue,very thick](n18) at (8.5,0.5) {$\tilde{z}_{n12}$};
    \node[blue,very thick](n99) at (9.5,-0.5) {$\tilde{z}_{112}$};
    \draw[densely dashed](n18)--(n99);
    \node[draw=none] (n151616) at (10.5,-1.5) {$z_{12}$};
    \draw[densely dashed](n99)--(n151616);
    \node[draw=none] (n92592962) at (13.5,-1.5) {$\cdots$};
    \node[red,very thick](n14) at (7,3) {$z_{32}$};
\end{tikzpicture}
\end{center}
\caption{Heronian frieze of order $n$}
\label{friz}
\end{figure}

Using the fact that measurements related to triangulations of a quadrilateral give rise to a Heronian diamond, the authors of \cite{fs} also state the following.

 \begin{figure}
\begin{center}
\begin{tikzpicture}
 \node[draw=none](n1) at (0,0) {$x_{il}$};
 \node[draw=none](n2) at (1,-1) {$S_{ijl}$};
 \node[draw=none](n3) at (2,-2) {$x_{jl}$};
 \node[draw=none](n4) at (1,-3) {$S_{jkl}$};
 \node[draw=none](n5) at (0,-4) {$x_{jk}$};
 \coordinate[label=${\mathbf{ijkl}}$] (c1) at (0,-2.4);
 \node[draw=none](n6) at (-1,-3){$S_{ijk}$};
 \node[draw=none](n7) at (-2,-2){$x_{ik}$};
 \node[draw=none](n8) at (-1,-1){$S_{ikl}$};
 \node[draw=none](n9) at (-2,-4){$x_{ij}$};
 \node[draw=none](n10) at (2,-4){$x_{kl}$};
 \node[draw=none](n11) at (-2,0){$x_{kl}$};
 \node[draw=none](n12) at (2,0){$x_{ij}$};
 
 \draw[-] (n1) to (n2);
 \draw[-] (n2) to (n3);
 \draw[-] (n3) to (n4);
 \draw[-] (n4) to (n5);
 \draw[-] (n5) to (n6);
 \draw[-] (n6) to (n7);
 \draw[-] (n7) to (n8);
 \draw[-] (n8) to (n1);
 \draw [densely dashed ] (n2) -- (n6);
 \draw[densely dashed] (n4) -- (n8);
 \draw[densely dashed] (n2) -- (n12);
 \draw[densely dashed] (n6) -- (n9);
 \draw[densely dashed] (n8) -- (n11);
 \draw[densely dashed] (n4) -- (n10);
\end{tikzpicture}
\end{center} 
\caption{A Heronian diamond for a quadruple of vertices $A_i, A_j, A_k, A_l$}
\label{fig3}
\end{figure}
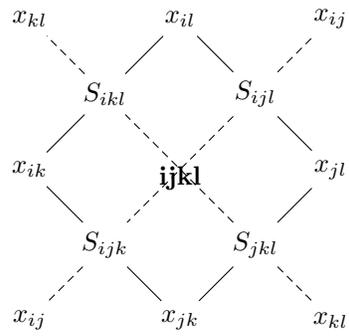 
\begin{prop}\cite{fs}
 Any $n$-gon $P$ in the complex plane gives rise to a Heronian frieze of order $n$, as given in the Figure \ref{friz}, in the following way:
\begin{equation*}
    z_{ij}:= x_{ij},
    \end{equation*}
    \begin{equation*}
    \tilde{z}_{i,i+1,j}:= S_{i,i+1,j},
\end{equation*}
\begin{equation*}
    \tilde{z}_{i,j,j+1}:= S_{i,j,j+1},
\end{equation*}
where $i, j \in \{1,2,...,n\}$, addition is modulo $n$, and the $x_{ij}$, $S_{i,i+1,j}$ and $S_{i,j,j+1}$ are as in \eqref{eqx} and \eqref{eqs}.\\
(Note that the boundary conditions \eqref{bound} hold since the squared distance between a vertex and itself equals zero. Similarly, the signed area of a triangle with two equal vertices is zero.)
\label{prop15}
\end{prop}
\begin{definition}
 Given an $n$-gon $P = (A_1,A_2,...,A_n)$, we define a \textit{polygonal Heronian frieze of order $n$} to be a Heronian frieze as in Proposition \ref{prop15}. Such a frieze is shown in the Figure \ref{fig6}.
\label{def16}
\end{definition}
\begin{remark}
Every diamond of a polygonal Heronian frieze corresponds to a choice of a quadruple of, not necessarily distinct, vertices of a corresponding $n$-gon. The diamonds of the first and the last row of the frieze correspond to quadruples where two of the four vertices coincide. Using the fact that each diamond has a label (see Remark \ref{labela}), the pattern of diamond labels of a polygonal Heronian frieze of order $n$ is given in  Figure 5 in \cite{key2}.
\end{remark}
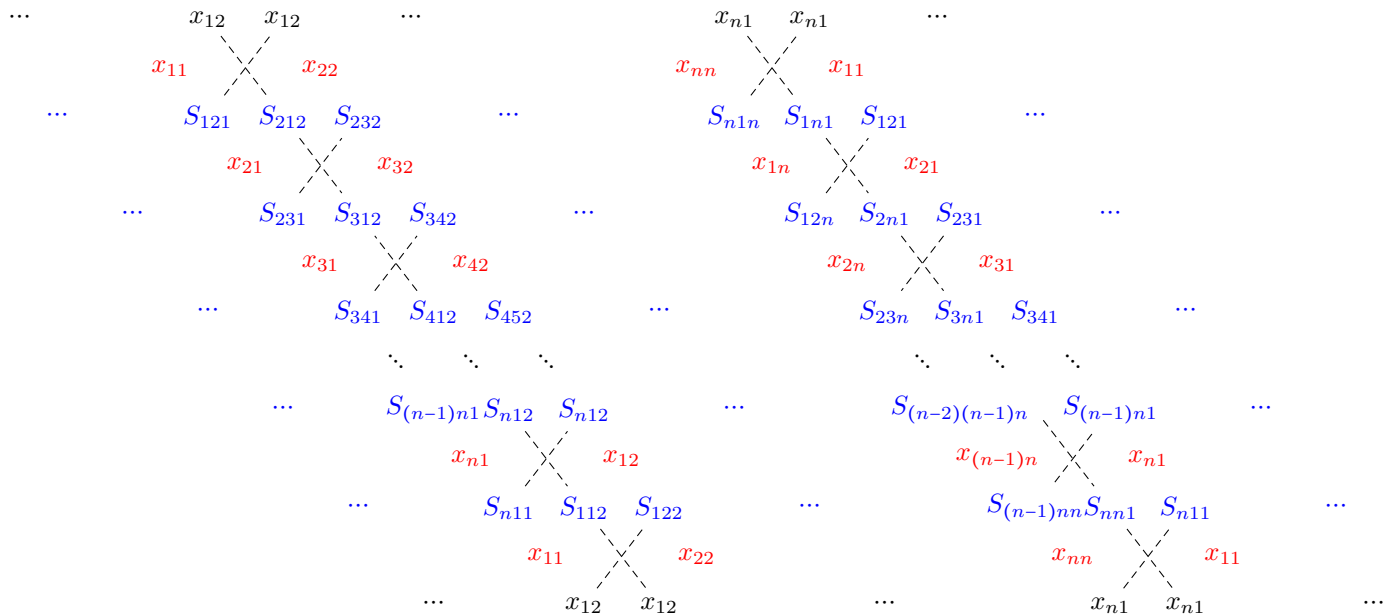
\begin{figure}
\begin{center}
\hspace*{-1.6cm}
\begin{tikzpicture}[xscale=1,yscale=1.3]
  \node[red,very thick](n1) at (4,4) {$x_{11}$};
   \node[blue,very thick](n2) at (4.5,3.5) {$S_{121}$};
   \node[blue,very thick](n15951951) at (2.5,3.5){$\cdots$};
   \node[draw=none] (n51621712) at (5.5,4.5) {$x_{12}$};
    \draw[densely dashed] (n2) -- (n51621712);
   \node[draw=none] (n51621712) at (4.5,4.5) {$x_{12}$};
\node[draw=none](n598515914151) at (2,4.5){$\cdots$};
   \node[draw=none] (n5262722) at (7.2,4.5) {$\cdots$};
    \node[red,very thick](n3) at (5,3) {$x_{21}$};
    \node[blue,very thick](n5) at (5.5,2.5) {$S_{231}$};
\node[blue,very thick](n591991) at (3.5,2.5){$\cdots$};
    \node[red,very thick](n4) at (6,2) {$x_{31}$};
    \node[blue,very thick](n55) at (6.5,1.5) {$S_{341}$};
    \node[blue,very thick](n8581581) at (4.5,1.5){$\cdots$};
    \node[red,very thick](n9) at (8,0) {$x_{n1}$};
    \node[draw=none](n12) at (7,1) {$\ddots$};
    
    \node[blue,very thick](n10) at (7.5,0.5) {$S_{(n-1)n1}$};
    \node[blue,very thick](n591561) at (5.5,0.5){$\cdots$};
    \node[blue,very thick](n7) at (8.5,-0.5) {$S_{n11}$};
   \node[blue,very thick](n419419) at (6.5,-0.5){$\cdots$};
  
    \node[red,very thick](n8) at (9,-1) {$x_{11}$};
    \node[draw=none] (n14151) at (9.5,-1.5) {$x_{12}$};
   \node[draw=none](n5188511) at (7.5,-1.5){$\cdots$};
    \node[red,very thick](n13) at (6,4) {$x_{22}$};
    
    \node[blue,very thick](n8043879) at (8.5,3.5) {$\cdots$};
    \node[blue,very thick](n17) at (6.5,3.5) {$S_{232}$};
\draw[densely dashed](n5) -- (n17);
    
    \node[blue,very thick](n87) at (7.5,2.5) {$S_{342}$};
    \draw[densely dashed](n55)--(n87);
    \node[blue,very thick](n956116) at (9.5,2.5){$\cdots$};
  \node[blue,very thick](n951991) at (12.5,2.5){$S_{12n}$};
  \node[blue,very thick](n1161996119) at (12.5,3.5){$S_{1n1}$};
  \node[blue,very thick](n599519961) at (11.5,3.5){$S_{n1n}$};
  \node[draw=none](n4194190) at (12.5,4.5){$x_{n1}$};
  \draw[densely dashed](n599519961)--(n4194190);
  \node[draw=none](n95159159) at (14.2,4.5){$\cdots$};
  \node[draw=none](n419151) at (11.5,4.5){$x_{n1}$};
  \draw[densely dashed](n419151)--(n1161996119);
  \node[red,very thick](n519519151) at (11,4){$x_{nn}$};
  \node[blue,very thick](n4949494491149) at (13.5,3.5){$S_{121}$};
  \draw[densely dashed](n951991) -- (n4949494491149);
  \node[blue,very thick](n599519859) at (15.5,3.5){$\cdots$};
  \node[red,very thick](n95159119) at (12,3){$x_{1n}$};
  \node[red,very thick](n5959595) at (13,2){$x_{2n}$};
  \node[red,very thick](n919599) at (13,4){$x_{11}$};
  \node[red,very thick](n4194149) at (14,3){$x_{21}$};
  \node[blue,very thick](n151951519) at (14.5,2.5){$S_{231}$};
  \node[blue,very thick](n915851591) at (16.5,2.5){$\cdots$};
  \node[red,very thick](n51519951) at (15,2){$x_{31}$};
  \node[blue,very thick](n595155151) at (15.5,1.5){$S_{341}$};
  \node[blue,very thick](n915115) at (17.5,1.5){$\cdots$};
  \node[draw=none](n591551166) at (16,1){$\ddots$};
  \node[blue,very thick](n51519518) at (16.5,0.5){$S_{(n-1)n1}$};
  \node[blue,very thick](n4185851) at (18.5,0.5){$\cdots$};
  \node[red,very thick](n591992592) at (17,0){$x_{n1}$};
  \node[blue,very thick](n949195) at (13.5,2.5){$S_{2n1}$};
  \draw[densely dashed](n1161996119) -- (n949195);
  \node[blue,very thick](n95151591) at (14.5,1.5){$S_{3n1}$};
  \node[draw=none](n95519196) at (15,1){$\ddots$};
  \draw[densely dashed](n949195) -- (n95151591);
  \node[blue,very thick](n95151541) at (13.5,1.5){$S_{23n}$};
  \draw[densely dashed](n151951519) -- (n95151541);
  \node[draw=none](n916916916) at (14,1){$\ddots$};
\node[blue,very thick](n915195141) at (14.5,0.5){$\tiny{S_{(n-2)(n-1)n}}$};
\node[blue,very thick](n59111581) at (15.5,0.5){};
\node[red,very thick](n599150) at (15,0){$x_{(n-1)n}$};
\node[blue,very thick](n95519581) at (15.5,-0.5){$S_{(n-1)nn}$};
\draw[densely dashed](n51519518) -- (n95519581);
\node[red,very thick](n699234113) at (16,-1){$x_{nn}$};
\node[draw=none](n91515151) at (16.5,-1.5){$x_{n1}$};
\node[blue,very thick](n59191591559) at (16.5,-0.5){ $ S_{nn1}$};
\draw[densely dashed](n59111581)--(n59191591559);
\node[blue,very thick](n95916661511) at (17.5,-0.5){$S_{n11}$};
\node[blue,very thick](n491951941) at (19.5,-0.5){$\cdots$};
\draw[densely dashed](n91515151)--(n95916661511);
\node[red,very thick](n544949449494) at (18,-1){$x_{11}$};
\node[draw=none](n499151161) at (17.5,-1.5){$x_{n1}$};
\node[draw=none](n519951915) at (20,-1.5){$\cdots$};
\draw[densely dashed](n59191591559) -- (n499151161);
    \node[red,very thick](n14) at (8,2) {$x_{42}$};
    \node[blue,very thick](n14) at (8.5,1.5) {$S_{452}$};
    \node[blue,very thick](n195551161) at (10.5,1.5){$\cdots$};
    \node[draw=none](n22) at (9,1) {$\ddots$};
    \node[blue,very thick](n24) at (9.5,0.5) {$S_{n12}$};
    \node[blue,very thick](n951815) at (11.5,0.5){$\cdots$};
    \draw[densely dashed] (n7) -- (n24); 
    \node[red,very thick](n65) at (10,0) {$x_{12}$};
    \node[blue,very thick](n32) at (10.5,-0.5) {$S_{122}$};
    \node[blue,very thick](n1816619) at (12.5,-0.5){$\cdots$};
     \draw[densely dashed] (n14151) -- (n32);
    \node[red,very thick](n90) at (11,-1) {$x_{22}$};
    \node[blue,very thick](n90) at (5.5,3.5) {$S_{212}$};
       \draw[densely dashed] (n51621712) -- (n90);
    \node[blue,very thick](n19) at (6.5,2.5) {$S_{312}$};
    \draw[densely dashed](n90)--(n19);
    \node[blue,very thick](n14) at (7.5,1.5) {$S_{412}$};
    \draw[densely dashed](n14)--(n19);
    \node[draw=none](n14) at (8,1) {$\ddots$};
    \node[blue,very thick](n18) at (8.5,0.5) {$S_{n12}$};
    \node[blue,very thick](n99) at (9.5,-0.5) {$S_{112}$};
    \draw[densely dashed](n18)--(n99);
    \node[draw=none] (n151616) at (10.5,-1.5) {$x_{12}$};
    \draw[densely dashed](n99)--(n151616);
    \node[draw=none] (n92592962) at (13.5,-1.5) {$\cdots$};
    \node[red,very thick](n14) at (7,3) {$x_{32}$};
\end{tikzpicture}
\end{center}
\caption{Polygonal Heronian frieze of order $n$}
\label{fig6}
\end{figure}
Now, we will establish what is needed to make a connection between minors of a certain matrix, which will be defined later in the text, and $S$ entries of a polygonal Heronian frieze. Then, we will use this connection to establish some relations that hold amongst the $S$ entries of a mentioned frieze. 
For that purpose, in the following two sections, we will be recalling some facts from \cite{key5}. 
\section{Exterior powers}
Let $V$ denote the complex vector space $\mathbb{C}^n$. Then $T(V)$ denotes the tensor algebra:
\begin{equation*}
    T(V) = V \oplus V \oplus V \oplus V \oplus ...
\end{equation*}
of $V$. The $\textit{exterior algebra}$ of $V$ is the quotient 
\begin{equation*}
    \bigwedge{(V)} = T(V)/J,
\end{equation*}
where $J$ denotes the ideal of $T(V)$ generated by elements $x \otimes x $ for $x \in V$. We denote the product in $\bigwedge(V)$ by $(x,y) \mapsto x \wedge y$ (recall that this comes from the tensor product in $T(V)).$ 
The $k$th exterior power, $\bigwedge^k(V)$ of $V$ is the subspace of $\bigwedge(V)$ spanned by the products of the form $v_1 \wedge v_2 \wedge \cdots \wedge v_k$, with $v_r \in V$ for $1 \leq r \leq k$. We have: 

\begin{equation*}
    \bigwedge(V) = \bigoplus_{k=0}^{\infty}{\bigwedge}^k{(V)}.
\end{equation*}

As usual, $e_1, e_2, ..., e_n$ denotes the natural basis of $V$, with $e_i(j) = \delta_{ij}$ for all $i,j$. Then the vectors 
\begin{equation*}
    e_{i_1}\wedge e_{i_2}\wedge \cdots \wedge e_{i_k},
\end{equation*}
with $1 \leq i_1 < i_2 < \cdots <i_k \leq n$ form a basis for $\bigwedge^k(V)$. For an element $x \in \bigwedge^k(V)$, we write the coefficient of $x$ when expended in terms of this basis by $p_{i_1,i_2,..,i_k}(x)$. Thus $p_{i_1,i_2,...,i_k}$ is a linear map from $\bigwedge^k(V)$ to $\mathbb{C}$. 

An element $x$ of $\bigwedge^k(V)$ is said to be \textit{decomposable} if it is of the form $x = v_1 \wedge v_2 \wedge ... \wedge v_k$ for some $k$, where $v_1, v_2, ..., v_k$ form a linearly independent set of vectors in $V$. Given such a linearly independent set $v_1,v_2,...,v_k$, we can form a $k \times n$ matrix $M$ of rank $k$ whose $i,j$ entry is $M_{ij} = (v_i)_j$, i.e. the $j$th entry of $v_i$, so that the rows are the vectors $v_i$. 
We denote the minor of $M$ corresponding to rows $a_1, a_2, ..., a_r$ and columns $b_1, b_2, ..., b_r$ by $\Delta_{b_1,b_2,...,b_r}^{a_1,a_2,...,a_r}(M)$. The following is well known and easy to check, using the expansion of the $v_i$ in terms of $e_1, e_2, ..., e_n$.
\begin{lemma}
Let $x = v_1 \wedge v_2 \wedge \cdots \wedge v_k$ be a decomposable element of $\bigwedge^k(V)$. Then $p_{i_1,i_2,...,i_k}(x)$ is the minor $\Delta_{i_1,i_2,...,i_k}^{1,2,...,k}$(M), where $M$ is the matrix defined above.
\label{lem11}
\end{lemma}
\section{The Grassmannian}
The \textit{Grassmannian} $Gr(k,n)$ is the set of $k$-dimensional subspaces of $V$. If $U$ is such a subspace and $v_1,v_2,...,v_k$ is a basis of $U$, consider the decomposable element 
\begin{equation*}
    w = v_1 \wedge v_2 \wedge \cdots \wedge v_k \in {\bigwedge}^k(V).
\end{equation*}
This element is non-zero and does not depend upon the choice of basis, up to a non-zero scalar. It follows that the tuple $(p_{i_1,i_2,...,i_k}(w))_{1\leq i_1 < i_2 < \cdots < i_k \leq n}$ is a well defined element of the projective space $\mathbb{P}^N$, where $N = \binom{n}{k} - 1$. Thus we obtain a map $\varphi: Gr(k,n) \rightarrow \mathbb{P}^N$. The entries $p_{i_1,i_2,...,i_k}$ are reffered to as \textit{Plücker coordinates}.

We extend the definition of $p_{i_1,i_2,...,i_k}$ to arbitrary elements $i_1,i_2,...,i_k$ in $I$ by setting $p_{i_1,i_2,...,i_k} = 0$ if $i_r = i_s$ for some $r,s$ and setting 
\begin{equation*}
    p_{i_1,i_2,...,i_k} = sgn(\pi)p_{j_1,j_2,...,j_k}
\end{equation*}
if the $i_r$ are distinct,
\begin{equation*}
    \{i_1,i_2,...,i_k\} = \{j_1, j_2, ..., j_k\},
\end{equation*}
$j_1<j_2<...<j_k$ and $\pi$ is the permutation:
\begin{equation*}
\begin{pmatrix}
i_1 & i_2 & \cdots & i_k\\
j_1 & j_2 & \cdots & j_k
\end{pmatrix}.
\end{equation*}
The Plücker relations for $Gr(k,n)$ are the relations:
\begin{equation}
\label{plucker}
    \sum_{r=0}^{k}(-1)^r p_{i_1,i_2,...,i_{k-1},j_r}p_{j_0,...,\hat{j_r},...,j_k} = 0,
\end{equation}
where the sum is taken over all tuples satisfying $1 \leq i_1 <i_2<...<i_{k-1}\leq n$ and $1 \leq j_0 <j_1<...<j_k\leq n$, and the hat indicates omission. \\ 
As pointed out in \cite{key6}, the order of the elements in the tuple plays a significant role. We will often need to use ordered tuples, or partially ordered tuples. Hence, we will use the following notation \cite{key6}. For a tuple $J = \{i_1,i_2,...,i_l\}$ of $[1,n]$ and $\{a_1,...a_l\} = \{i_1,...,i_l\}$ with $1 \leq a_1 \leq ... a_l \leq n$ we set 
\begin{equation*}
    o(J) = o(i_1,...,i_l) = (a_1,...,a_l).
\end{equation*}
With this notation, if $0 \leq s,l \leq k$, we will write $p_{b_1,...,b_s,o(i_1,...,i_l),c_1,...,c_{k-l+a}}$ for $p_{b_1,...,b_s,a_1,...,a_l,c_1,...,c_{k-l+a}}$. \\ 
 Now we have all the prerequsities needed for establishing the above mentioned connection.
\section{Heronian minor relations}
\begin{definition}
Let $A$ be an $m \times n$ matrix. Then we define $m_{i_1i_2...i_m}$ to be a minor of the matrix $A$ that corresponds to the submatrix whose columns are $i_1,i_2,...,i_m$.
\end{definition}
\begin{lemma}
\label{lem1}
Let $F$ be a polygonal Heronian frieze rising from an $n$-gon $P = (A_1, A_2,..., A_n)$, where $A_m = (x_m, y_m)$, for $m \in \{1, 2, ..., n\}$, and let $C$ be the $3 \times  
 n$ matrix whose $m$-th column is $(1, x_m, y_m)$, for $m \in \{1,2, ..., n\}$. Then the minor $m_{i_1i_2i_3}$ is equal to the two times signed area of the triangle whose vertices are $A_{i_1}$, $A_{i_2}$ and $A_{i_3}$.
\end{lemma}
\begin{proof}
Follows easily from the very well known formula for the signed area of the triangle.
\end{proof}
\begin{definition} 
Let $P$, $F$ and $C$ be a polygon, Heronian frieze, and matrix, respectively,  as in Lemma \ref{lem1}. Then we say that $C$ is the \textit{coordinate matrix} of $F$ (or $P$).
\end{definition}
 Since, as stated earlier, the $S$ entries of the polygonal Heronian frieze represent 4 times the  signed areas of triangles, it follows that some $3 \times 3$ minors of the matrix $C$ give rise to the $S$ entries of the frieze-namely, two times minor $m_{ijk}$ equals the $S_{ijk}$ entry of the frieze, i.e.
 \begin{equation*}
     2  m_{ijk} = S_{ijk}.
 \end{equation*}
 However, each $S$ entry of the frieze is either of the form $S_{a,a+1,b}$, or $S_{a,b,b+1}$, where $a,b \in \{1,2,...,n\}$ and addition is modulo $n$. It follows that there is a restriction on the $3 \times 3$ minors of matrix $C$  giving rise to the $S$ entries of the frieze. Hence the following definition.
\begin{definition}
\label{her_minor}
Let $F$ be a polygonal Heronian frieze rising from an $n$-gon $P=(A_1,A_2,...,A_n)$. If a minor of the coordinate matrix  $C$ is in one of the following forms \begin{itemize}
 \item $m_{a,a+1,b}$
 \item $m_{a,b,b+1}$,
 \end{itemize}
 where $a,b \in \{1,2,...,n\}$ and addition is modulo $n$, 
then we say that it is a \textit{Heronian minor}.
\end{definition}
The signifance of relating $S$ entries of the frieze $F$ with minors of the matrix $C$ lies in the fact that the former satisfy certain relations. Hence the following theorem. 
\begin{theorem} Let $C$ be a $3 \times n$ coordinate matrix of the polygonal Heronian frieze $F$. Then the following holds for its minors:
\begin{equation}\label{eqn_plucker}
 \sum_{r=0}^{3} (-1)^r m_{i_1,i_2, j_r}m_{j_0,...,\hat{j_r},...,j_3} = 0,
\end{equation}
where the sum is taken over all the tuples satisfying $1 \leq i_1 < i_2 \leq n$ and $1 \leq j_0 < j_1 < j_2 <j_3 \leq n$, and the hat indicates omission. 
\end{theorem}
\begin{proof}It follows from \eqref{plucker} that the Plücker relations for $Gr(3,n)$ are the following
\begin{equation*}
  \sum_{r=0}^{3} (-1)^r p_{i_1,i_2, j_r}p_{j_0,...,\hat{j_r},...,j_3} = 0,
\end{equation*}
where the sum is taken over all tuples satisfying $1 \leq i_1 < i_2 \leq n$ and $1 \leq j_0 < j_1 < j_2 <j_3 \leq n$, and the hat indicates omission.
Since Plücker coordinates are defined in terms of minors (see Lemma \ref{lem11}), the statement follows.
\end{proof}
\begin{definition}
\label{her_mnr_rln}
Let $F$ be a polygonal Heronian frieze of order $n$, and let $C$ be its coordinate matrix. If relation \eqref{eqn_plucker} is such that all the minors appearing in it are Heronian (up to a sign), then we say that it is a \textit{Heronian minor relation}.
\end{definition}
\begin{remark} The only minors appearing in the relation, which are the negatives of Heronian minors, are those of the form $m_{1,n,b}$, where $b \in \{1,...,n\}$.
\end{remark}
 Having assumed that we are given a polygonal Heronian frieze $F$ and its coordinate matrix $C$, we have that Heronian minor relations give relations between some of the $S$ entries of the frieze.\\ Namely, if \eqref{eqn_plucker} is a Heronian minor relation, where $(i_1,i_2)=(1, n)$, then by multiplying it by $4$, we get
 \begin{equation}
     -S_{n,1,j_0}S_{j_1,j_2,j_3}+S_{n,1,j_1}S_{j_0,j_2,j_3}-S_{n,1,j_2}S_{j_0,j_1,j_3}+S_{n,1,j_3}S_{j_0,j_1,j_2} = 0,
     \label{srln}
 \end{equation}
 in case $(i_1,i_2)=(1,n)$, and 
  \begin{equation}
     S_{i_1i_2j_0}S_{j_1j_2j_3}-S_{i_1i_2j_1}S_{j_0j_2j_3}+S_{i_1i_2j_2}S_{j_0j_1j_3}-S_{i_1i_2j_3}S_{j_0j_1j_2} = 0,
     \label{srln1}
 \end{equation}
 in case $(i_1,i_2) \neq (1,n)$.
\begin{definition}
Having a Heronian minor relation \eqref{eqn_plucker}, we will be referring to, depending on the case, either \eqref{srln} or \eqref{srln1}, as to  \textit{the corresponding $S$-relation}.
\end{definition}
\begin{theorem} 
Let $F$ be a polygonal Heronian frieze of order $n$, and let $C$ be its coordinate matrix. Assume that the choice of tuples $1 \leq i_1 < i_2 \leq n$ and $1 \leq j_0 < j_1 < j_2 <j_3 \leq n$ in \eqref{eqn_plucker} is such that \\ $(i_1,i_2) \notin \{(j_0,j_1),(j_1,j_2),(j_2,j_3),(j_0,j_3)\}$. \\ \\
(a) If $\{i_1,i_2\} \cap \{j_0,j_1,j_2,j_3\} = \emptyset$, then \eqref{eqn_plucker} is a Heronian minor relation if and only if one of the following holds:
\begin{itemize} 
\item
\begin{equation*}
i_1=1, i_2=n, j_1=j_0+1,j_3=j_2+1,
\end{equation*}
and the corresponding $S$-relation is \begin{equation*}
    S_{n,1,j_0}S_{j_0+1,j_2,j_2+1}-S_{n,1,j_0+1}S_{j_0,j_2,j_2+1}+S_{n,1,j_2}S_{j_0,j_0+1,j_2+1}-S_{n,1,j_2+1}S_{j_0,j_0+1,j_2} = 0,
\end{equation*}
where $j_0<j_2-1$.
\item  \begin{equation*}
    1\leq i_1 \leq n-1, i_2=i_1+1, j_1=j_0+1, j_3=j_2+1,
\end{equation*}
and the corresponding $S$-relation is 
\begin{equation*}
   S_{i_1,i_1+1,j_0}S_{j_0+1,j_2,j_2+1}-S_{i_1,i_1+1,j_0+1}S_{j_0,j_2,j_2+1}+S_{i_1,i_1+1,j_2}S_{j_0,j_0+1,j_2+1}-S_{i_1,i_1+1,j_2+1}S_{j_0,j_0+1,j_2}=0,
\end{equation*}
where $j_0<j_2-1$.

\item 
\begin{equation*}
1 < i_1 < n-1, i_2=i_1+1, j_0=1, j_2=j_1+1, j_3 =n,
\end{equation*}
and the corresponding $S$-relation is 
\begin{equation*}
    S_{i_1,i_1+1,1}S_{j_1,j_1+1,n}-S_{i_1,i_1+1,j_1}S_{n,1,j_1+1}+S_{i_1,i_1+1,j_1+1}S_{n,1,j_1}-S_{i_1,i_1+1,n}S_{1,j_1,j_1+1}=0,
\end{equation*}
where $1<j_1<n$.
\end{itemize}
(b) If $\{i_1,i_2\} \cap \{j_0,j_1,j_2,j_3\}=\{j_0\}$, then \eqref{eqn_plucker} is a Heronian minor relation if and only if one of the following holds:
\begin{itemize}
\item 
\begin{equation*}
    i_1=1, i_2=2, j_0=1, j_2=j_1+1, j_3=n, 
\end{equation*}
and the corresponding $S$-relation is
\begin{equation*}
    -S_{1,2,j_1}S_{n,1,j_1+1}+S_{1,2,j_1+1}S_{n,1,j_1}-S_{1,2,n}S_{1,j_1,j_1+1}=0,
\end{equation*}
where $1<j_1<n$.
\item 
\begin{equation*}
    i_1=1, i_2=n, j_0=1, j_1=2, j_3=j_2+1,
\end{equation*}
and the corresponding $S$-relation is 
\begin{equation*}
    S_{n,1,2}S_{1,j_2,j_2+1}-S_{n,1,j_2}S_{1,2,j_2+1}+S_{n,1,j_2+1}S_{1,2,j_2}=0,
\end{equation*}
where $2<j_2<n$.
\item 
\begin{equation*}
    1 \leq i_1 \leq n-1, i_2=i_1+1, j_0=i_1+1, j_1=i_1+2, j_3=j_2+1,
\end{equation*}
and the corresponding $S$-relation is 
\begin{equation*}
    S_{i_1,i_1+1,i_1+2}S_{i_1+1,j_2,j_2+1}+S_{i_1,i_1+1,j_2}S_{i_1+1,i_1+2,j_2+1}-S_{i_1,i_1+1,j_2+1}S_{i_1+1,i_1+2,j_2}=0,
\end{equation*}
where $i_1+2<j_2<n$.
\end{itemize}
(c) If $\{i_1,i_2\} \cap \{j_0,j_1,j_2,j_3\}=\{j_1\}$, then \eqref{eqn_plucker} is a Heronian minor relation if and only if one of the following holds:
\begin{itemize}
    \item \begin{equation*}
        1 < i_1 < n-1, i_2=i_1+1, j_0=1, j_1=i_1+1, j_2=i_1+2, j_3=n,
    \end{equation*}
    and the corresponding $S$-relation is
    \begin{equation*}
        S_{i_1,i_1+1,1}S_{i_1+1,i_1+2,n}+S_{i_1,i_1+1,i_1+2}S_{n,1,i_1+1}-S_{i_1,i_1+1,n}S_{1,i_1+1,i_1+2}=0,
    \end{equation*}
    where $1<i_1<n-1$.
    \item \begin{equation*}
        1 < i_1 \leq n-1, i_2=i_1+1, j_0=i_1-1, j_1=i_1, j_3=j_2+1,
    \end{equation*}
    and the corresponding $S$-relation is 
    \begin{equation*}
        S_{i_1,i_1+1,i_1-1}S_{i_1,j_2,j_2+1}+S_{i_1,i_1+1,j_2}S_{i_1-1,i_1,j_2+1}-S_{i_1,i_1+1,j_2+1}S_{i_1-1,i_1,j_2}=0,
    \end{equation*}
    where $1<i_1<j_2$.
\end{itemize}
(d) If $\{i_1,i_2\} \cap \{j_0,j_1,j_2,j_3\}=\{j_2\}$, then \eqref{eqn_plucker} is a Heronian minor relation if and only if one of the following holds:
\begin{itemize}
    \item 
    \begin{equation*}
        1 < i_1 < n-1, i_2=i_1+1, j_0=1, j_1=i_1-1, j_2=i_1, j_3=n,
    \end{equation*}
    and the corresponding $S$-relation is 
    \begin{equation*}
    S_{i_1,i_1+1,1}S_{i_1-1,i_1,n}-S_{i_1,i_1+1,i_1-1}S_{n,1,i_1}-S_{i_1,i_1+1,n}S_{1,i_1-1,i_1}=0,
    \end{equation*}
    where $1 < i_1 < n-1$.
    \item 
    \begin{equation*}
        1 \leq i_1 < n-1, i_2=i_1+1, j_1=j_0+1, j_2=i_1+1, j_3=i_1+2,
    \end{equation*}
    and the corresponding $S$-relation is 
    \begin{equation*}
        S_{i_1,i_1+1,j_0}S_{j_0+1,i_1+1,i_1+2}-S_{i_1,i_1+1,j_0+1}S_{j_0,i_1+1,i_1+2}-S_{i_1,i_1+1,i_1+2}S_{j_0,j_0+1,i_1+1}=0,
    \end{equation*}
    where $j_0<i_1$.
\end{itemize}
(e) If $\{i_1,i_2\} \cap \{j_0,j_1,j_2,j_3\} = \{j_3\}$, then \eqref{eqn_plucker} is a Heroninan minor relation if and only if one of the following holds:
\begin{itemize}
    \item 
    \begin{equation*}
        i_1=n-1, i_2=n, j_0=1, j_2=j_1+1, j_3=n,
    \end{equation*}
    and the corresponding $S$-relation is
    \begin{equation*}
        S_{n-1,n,1}S_{j_1,j_1+1,n}-S_{n-1,n,j_1}S_{n,1,j_1+1}+S_{n-1,n,j_1+1}S_{n,1,j_1} = 0,
    \end{equation*}
    where $1<j_1<n$.
    \item 
    \begin{equation*}
        i_1=1, i_2=n, j_1=j_0+1, j_2=n-1, j_3=n,
    \end{equation*}
    and the corresponding $S$-relation is 
    \begin{equation*}
        S_{n,1,j_0}S_{j_0+1,n-1,n}-S_{n,1,j_0+1}S_{j_0,n-1,n}+S_{n,1,n-1}S_{j_0,j_0+1,n}=0,
    \end{equation*}
    where $j_0<n-2$.
    \item 
    \begin{equation*}
        1 < i_1 \leq n-1, i_2=i_1+1, j_1=j_0+1, j_2=i_1-1, j_3 = i_1,
    \end{equation*}
    and the corresponding $S$-relation is 
    \begin{equation*}
        S_{i_1,i_1+1,j_0}S_{j_0+1,i_1-1,i_1}-S_{i_1,i_1+1,j_0+1}S_{j_0,i_1-1,i_1} + S_{i_1,i_1+1,i_1-1}S_{j_0,j_0+1,i_1}=0,
    \end{equation*}
    where $j_0<i_1-2$.
\end{itemize}
\label{main}
\end{theorem}
\begin{proof}  
Proof of (a): Assume that $\{i_1,i_2\} \cap \{j_0,j_1,j_2,j_3\} = \emptyset$. It is easy to see that the first parts of the summands of \eqref{eqn_plucker} are Heronian if and only if $1 \leq i_1 \leq n-1$ and $i_2=i_1+1$, or $i_1=1$ and $i_2=n$. Let us label this condition with $I$, i.e.
\begin{equation}
    I: = [(1 \leq i_1 \leq n-1 \land i_2=i_1+1) \lor (i_1=1 \land i_2=n)].
    \label{uslovvv}
\end{equation}
It is also not difficult to verify that the minors $m_{j_0,...,\hat{j_r},....,j_3}$, for $r \in \{0,1,2,3\}$, appearing in the sum \eqref{eqn_plucker} are Heronian if and only if the choice of the tuple $1 \leq j_0<j_1<j_2<j_3\leq n$ is such that 
\begin{multline}
\label{condition}
    (j_2 = j_1+1 \lor j_3 = j_2+1) \land ( j_3 = j_2+1 \lor j_0=j_3+1)\\ \land (j_1 = j_0+1 \lor j_0=j_3+1) \land (j_1 = j_0+1 \lor j_2 = j_1+1),
\end{multline}
 where equalities are modulo $n$. \\ \\ 
Now, for the ease of reference, let us introduce the following notation for the statement $j_{s}=j_{r}+1$ $(mod$ $n)$:
\begin{equation*}
     J_{r,s} : \quad j_{r}+1 = j_s\quad (mod \quad n).
    \end{equation*} 
Note that, in the case $(r,s) \neq (3,0)$, $J_{r,s}$ is equivalent to $j_{r}+1=j_s$, and in the case $(r,s)=(3,0)$, $J_{r,s}$ is equivalent to $j_3=n$ and $j_0=1$.\\ \\ 
With this notation, \eqref{condition} becomes:
\begin{equation}
    (J_{1,2} \lor J_{2,3}) \land ( J_{2,3} \lor J_{3,0}) \land (J_{0,1}  \lor J_{3,0}) \land (J_{0,1} \lor J_{1,2}).
    \label{eqnnn}
\end{equation}
It follows that \eqref{eqn_plucker} is a Heronian minor relation if and only if $I$ and 
\eqref{eqnnn}.\\ The statement \eqref{eqnnn} holds true if and only if at least one of the sixteen conjunctions of the fifth column of  Table \ref{table1} holds true.  
\begin{table}[h!]
 \centering
 \begin{tabular}{|c|c|c|c|c|c|c|c|c|c|c|}
 \hline
 $\mathbf{J_{1,2} \lor J_{2,3}}$ & $\mathbf{J_{2,3} \lor J_{30}}$ & $\mathbf{J_{0,1} \lor  J_{30}}$ & $\mathbf{J_{0,1} \lor J_{1,2}}$  & \textbf{conjuction}  \\
 \hline 
   $J_{1,2}$ & $J_{2,3}$ & $J_{0,1}$ & $J_{0,1}$ & $J_{0,1}\land J_{1,2}\land J_{2,3}$ 
 \\ \hline
 $J_{1,2}$ & $J_{2,3}$ & $J_{0,1}$ & $J_{1,2}$ & $J_{0,1}\land J_{1,2} \land J_{2,3}$
 \\ \hline
 $J_{1,2}$ & $J_{2,3}$ & $J_{3,0}$ & $J_{0,1}$ & $J_{0,1} \land J_{1,2} \land J_{2,3} \land J_{3,0}$ 
 \\ \hline
  $J_{1,2}$ & $J_{2,3}$ & $J_{3,0}$ & $J_{1,2}$ & $J_{1,2} \land J_{2,3} \land J_{3,0}$  \\
 \hline
 $J_{1,2}$ & $J_{3,0}$ & $J_{0,1}$ & $J_{0,1}$ & $J_{0,1} \land J_{1,2} \land J_{3,0} $ \\
 \hline
 $J_{1,2}$ & $J_{3,0}$ & $J_{0,1}$ & $J_{1,2}$ &   $J_{0,1} \land J_{1,2} \land J_{3,0}$ \\
 \hline
 $J_{1,2}$ & $J_{3,0}$ & $J_{3,0}$ & $J_{0,1}$ & $J_{0,1} \land J_{1,2} \land J_{3,0}$ \\ 
 \hline
 $J_{1,2}$ & $J_{3,0}$ & $J_{3,0}$ & $J_{1,2}$ & $J_{1,2}  \land J_{3,0}$  \\
 \hline
 $J_{2,3}$ & $J_{2,3}$ & $J_{0,1}$ & $J_{0,1}$ & $J_{0,1} \land J_{2,3}$ \\
 \hline
 $J_{2,3}$ & $J_{2,3}$ & $J_{0,1}$ & $J_{1,2}$ & $J_{0,1} \land J_{1,2} \land J_{2,3}$  \\
 \hline
 $J_{2,3}$ & $J_{2,3}$ & $J_{3,0}$ & $J_{0,1}$ & $J_{0,1} \land J_{2,3} \land J_{3,0}$  \\
 \hline
 $J_{2,3}$ & $J_{2,3}$ & $J_{3,0}$ & $J_{1,2}$ & $J_{1,2} \land J_{2,3} \land J_{3,0}$  \\
 \hline
 $J_{2,3}$ & $J_{3,0}$ & $J_{0,1}$ & $J_{0,1}$ & $J_{0,1} \land J_{2,3} \land J_{3,0}$ \\
 \hline 
 $J_{2,3}$ & $J_{3,0}$ & $J_{0,1}$ & $J_{1,2}$ & $J_{0,1} \land J_{1,2} \land J_{2,3} \land J_{3,0}$  \\
 \hline
 $J_{2,3}$ & $J_{3,0} $& $J_{3,0}$ & $J_{0,1}$ & $J_{0,1} \land J_{2,3} \land J_{3,0} $  \\
 \hline
 $J_{2,3}$ & $J_{3,0} $& $J_{3,0}$ & $J_{1,2}$ & $J_{1,2} \land J_{2,3} \land J_{3,0}$  \\
 \hline
 \end{tabular}
 \caption{}
 \label{table1}
 \end{table} 
From Table \ref{table1}, it follows that, if $\{i_1,i_2\} \cap \{j_0,j_1,j_2,j_3\} = \emptyset$, then \eqref{eqnnn} holds true if and only if 
 \begin{equation*}
     ( J_{0,1} \land J_{1,2} \land J_{2,3}) \lor ( J_{0,1} \land J_{2,3}) \lor (J_{12} \land J_{30}) \lor (J_{01} \land J_{23} \land J_{30}).
 \end{equation*}
 Hence, in case $\{i_1,i_2\} \cap \{j_0,j_1,j_2,j_3\} = \emptyset$, \eqref{eqn_plucker} is a Heronian minor relation if and only if 
 \begin{equation*}
     I \land [(J_{0,1} \land J_{1,2} \land J_{2,3}) \lor (J_{0,1} \land J_{2,3}) \lor (J_{12} \land J_{30}) \lor (J_{01} \land J_{23} \land J_{30})],
 \end{equation*}
 which is equivalent to $(I \land J_{0,1} \land J_{2,3}) \lor (I \land J_{1,2} \land J_{3,0})$. \\ \\
 If $I \land J_{0,1} \land J_{2,3}$, then we have either 
 \begin{itemize}
     \item 
 $i_1=1, i_2=n, j_1=j_0+1, j_3=j_2+1$; \\  \\ 
 the Heronian minor relation we get is 
 \begin{equation*}
     m_{1,n,j_o}m_{j_0+1,j_2,j_2+1}-m_{1,n,j_0+1}m_{j_0,j_2,j_2+1}+m_{1,n,j_2}m_{j_0,j_0+1,j_2+1}-m_{1,n,j_2+1}m_{j_0,j_0+1,j_2}=0,
 \end{equation*}
 and the corresponding $S$-relation is 
  \begin{equation*}
     S_{n,1,j_o}S_{j_0+1,j_2,j_2+1}-S_{n,1,j_0+1}S_{j_0,j_2,j_2+1}+S_{n,1,j_2}S_{j_0,j_0+1,j_2+1}-S_{n,1,j_2+1}S_{j_0,j_0+1,j_2}=0,
 \end{equation*}
 where $j_0<j_2-1$,  \\ \\  or 
 \item $1 \leq i_1 \leq n-1, i_2=i_1+1, j_1=j_0+1, j_3=j_2+1$; \\ \\ the Heronian minor relation we get is 
 \begin{equation*}
     m_{i_1,i_1+1,j_o}m_{j_0+1,j_2,j_2+1}-m_{i_1,i_1+1,j_0+1}m_{j_0,j_2,j_2+1}+m_{i_1,i_1+1,j_2}m_{j_0,j_0+1,j_2+1}-m_{i_1,i_1+1,j_2+1}m_{j_0,j_0+1,j_2}=0,
 \end{equation*}
 and the corresponding $S$-relation is 
  \begin{equation*}
         S_{i_1,i_1+1,j_o}S_{j_0+1,j_2,j_2+1}-S_{i_1,i_1+1,j_0+1}S_{j_0,j_2,j_2+1}+S_{i_1,i_1+1,j_2}S_{j_0,j_0+1,j_2+1}-S_{i_1,i_1+1,j_2+1}S_{j_0,j_0+1,j_2}=0, 
 \end{equation*}
 where $j_0<j_2-1$.
 \end{itemize}
 In the similar manner, if $I \land J_{1,2} \land J_{3,0}$, then we have either 
 \begin{itemize}
     \item $i_1=1,i_2=n,j_0=1,j_2=j_1+1, j_3=n$, which leads to contradiction, since we assumed that $\{i_1,i_2\} \cap \{j_0,j_1,j_2,j_3\} = \emptyset$, \\ \\ or
     \item $1< i_1 < n-1, i_2=i_1+1, j_0=1, j_2=j_1+1, j_3=n$. \\ \\
      After using that $m_{1,j_1+1,n}=m_{n,1,j_1+1}, m_{1,j_1,n}=m_{n,1,j_1}$, the Heronian minor relation we get is 
  \begin{equation*}
     m_{i_1,i_1+1,1}m_{j_1,j_1+1,n}-m_{i_1,i_1+1,j_1}m_{n,1,j_1+1}+m_{i_1,i_1+1,j_1+1}m_{n,1,j_1}-m_{i_1,i_1+1,n}m_{1,j_1,j_1+1}=0,
 \end{equation*}
 and the corresponding $S$-relation is 
  \begin{equation*}
         S_{i_1,i_1+1,1}S_{j_1,j_1+1,n}-S_{i_1,i_1+1,j_1}S_{n,1,j_1+1}+S_{i_1,i_1+1,j_1+1}S_{n,1,j_1}-S_{i_1,i_1+1,n}S_{1,j_1,j_1+1}=0,
 \end{equation*}
 where $1<j_1<n$,
 \end{itemize}
This completes proof of part (a). \\ \\ Proof of (b).
  Let us now assume that $\{i_1,i_2\} \cap \{j_0,j_1,j_2,j_3\} = \{j_0\}$. Then $m_{i_1i_2j_0} = 0$, so \eqref{eqn_plucker} becomes 
 \begin{equation}
     -m_{i_1i_2j_1}m_{j_0j_2j_3} + m_{i_1i_2j_2}m_{j_0j_1j_3}-m_{i_1i_2j_3}m_{j_0j_1j_2} = 0.
     \label{four}
 \end{equation}
As in the previous case, we conclude that the first parts of the summands of \eqref{eqn_plucker} are Heronian if and only if \eqref{uslovvv} holds. Furthermore, minors $m_{j_0,...,\hat{j_r},....,j_3}$, for $r \in \{0,1,2,3\}$ appearing in \eqref{four} are Heronian if and only if the choice of the tuple $1\leq j_0<j_1<j_2<j_3\leq n$ is such that 
\begin{equation}
    ( J_{2,3} \lor J_{3,0}) \land (J_{0,1} \lor   J_{3,0}) \land (J_{0,1} \lor J_{1,2}).
    \label{five}
\end{equation}
It follows that, in case $\{i_1,i_2\} \cap \{j_0,j_1,j_2,j_3\} = \{j_0\}$, \eqref{eqn_plucker} is a Heronian minor relation if and only if $I$ and \eqref{five}. Statement \eqref{five} holds if and only if at least one of the eight conjuctions of the fourth column of Table \ref{table2} holds true.
\begin{table}[h!]
 \centering
 \begin{tabular}{|c|c|c|c|c|c|c|c|c|c|c|}
 \hline
 $\mathbf{ J_{2,3} \lor J_{3,0}}$ & $\mathbf{J_{0,1} \lor J_{3,0}}$ &  $\mathbf{J_{0,1} \lor J_{1,2} }$  & \textbf{conjuction }   \\
 \hline 
   $J_{2,3}$ & $J_{0,1}$  & $J_{0,1}$ & $J_{0,1}\land J_{2,3}$ 
 \\ \hline
 $J_{2,3}$ & $J_{0,1}$ &  $J_{1,2}$ & $J_{0,1}\land J_{1,2} \land J_{2,3}$
 \\ \hline
 $J_{2,3}$ & $J_{3,0}$  & $J_{0,1}$ & $J_{0,1} \land J_{2,3} \land J_{3,0} $ 
 \\ \hline
  $J_{2,3}$ & $J_{3,0}$ &  $J_{1,2}$ & $J_{1,2} \land J_{2,3} \land J_{3,0} $  \\
 \hline
 $J_{3,0}$ & $J_{0,1}$  & $J_{0,1}$ &   $J_{0,1} \land J_{3,0}  $ \\
 \hline
 $J_{3,0}$ & $J_{0,1}$ &  $J_{1,2}$ & $J_{0,1} \land J_{1,2} \land J_{3,0}$ \\ 
 \hline
 $J_{3,0}$ & $J_{3,0}$ &  $J_{0,1}$ & $J_{0,1} \land J_{3,0}$ \\
 \hline
  $J_{3,0}$ & $J_{3,0}$ &  $J_{1,2}$ & $J_{1,2} \land J_{3,0}$  \\
 \hline
 \end{tabular}
 \caption{}
 \label{table2}
 \end{table} 
From Table \ref{table2}, it follows that (noting that we have assumed that $\{i_1,i_2\} \cap \{j_0,j_1,j_2,j_3\}=\{j_0\}$) \eqref{five} holds true if and only if 
\begin{multline*}
     ( J_{0,1} \land J_{2,3}) \lor ( J_{0,1} \land J_{1,2}\land J_{2,3}) \lor ( J_{0,1} \lor J_{2,3} \lor J_{3,0} ) \\  \lor (J_{1,2} \land J_{2,3} \land J_{3,0}) \lor (J_{0,1} \land J_{3,0}) \lor (J_{0,1} \land J_{1,2} \land J_{3,0}) \lor  (J_{1,2} \land J_{3,0}).
\end{multline*}
Hence, in case $\{i_1,i_2\} \cap \{j_0,j_1,j_2,j_3\}=\{j_0\}$, \eqref{eqn_plucker} is a Heronian minor relation if and only if 
\begin{multline*}
    I \land [( J_{0,1} \land J_{2,3}) \lor ( J_{0,1} \land J_{1,2}\land J_{2,3}) \lor ( J_{0,1} \lor J_{2,3} \lor J_{3,0} ) \\  \lor (J_{1,2} \land J_{2,3} \land J_{3,0}) \lor (J_{0,1} \land J_{3,0}) \lor (J_{0,1} \land J_{1,2} \land J_{3,0}) \lor  (J_{1,2} \land J_{3,0})
     ],
\end{multline*}
which is equivalent to 
\begin{equation*}
    (I\land J_{0,1} \land J_{3,0}) \lor (I \land J_{1,2} \land J_{3,0}) \lor (I \land J_{0,1} \land J_{2,3}).
\end{equation*}
If $I \land J_{0,1} \land J_{3,0}$, then we have either 
\begin{itemize}
\item $i_1=1,i_2=n,j_3=n,j_0=1, j_1=j_0+1=2$, which leads to a contradiction, due to the assumption that $\{i_1,i_2\} \cap \{j_0,j_1,j_2,j_3\} = \{j_0\},$ \\ \\ or
\item $1\leq i_1 \leq n-1, i_2=i_1+1, j_3=n, j_0=1, j_1= j_0+1=2$. \\ \\
Since $\{i_1,i_2\} \cap \{j_0,j_1,j_2,j_3\} = \{j_0\}$, it follows that either $i_1=1$, or $i_2=1$, and each of these cases leads to contradiction -  due to the facts that  $\{i_1,i_2\} \cap \{j_0,j_1,j_2,j_3\} = \{j_0\}$, and $i_1 \geq 1$, respectively.
\end{itemize}
In the case $I \land J_{1,2} \land J_{3,0}$, we have either 
\begin{itemize}
\item $i_1=1, i_2=n, j_3=n, j_0=1, j_2=j_1+1$, which leads to contradiction, since $\{i_1,i_2\} \cap \{j_0,j_1,j_2,j_3\}=\{j_0\},$ \\ \\ or
\item $1 \leq i_1 \leq n-1, i_2=i_1+1, j_3=n, j_0=1, j_2=j_1+1$. \\ Since $\{i_1,i_2\} \cap \{j_0,j_1,j_2,j_3\} = \{j_0\}$, it follows that either $i_1=1$ or $i_2=1$.\\ $i_2=1$ leads to contradiction, so we have that $i_1=1$.
\\ \\ 
After using that $m_{1,j_1+1,n} = m_{n,1,j_1+1}$ and $m_{1,j_1,n}=m_{n,1,j_1}$, the Heronian minor relation we get is 
\begin{equation*}
   -m_{1,2,j_1}m_{n,1,j_1+1}+m_{1,2,j_1+1}m_{n,1,j_1}-m_{1,2,n}m_{1,j_1,j_1+1}=0,
\end{equation*}
and the corresponding $S$-relation is 
\begin{equation*}
   -S_{1,2,j_1}S_{n,1,j_1+1}+S_{1,2,j_1+1}S_{n,1,j_1}-S_{1,2,n}S_{1,j_1,j_1+1}=0,
\end{equation*}
where $1<j_1<n$. 
\end{itemize}
We proceed similarly in the case $I \land J_{0,1} \land J_{2,3}$. Namely, we have either
\begin{itemize}
    \item $i_1=1,i_2=n,j_3=j_2+1,j_1=j_0+1$. \\ 
    Since $\{1,n\} \cap \{j_0,j_1,j_2,j_3\}=\{j_0\}$, and $j_0 < n$, we conclude that $j_0=1$, and hence $j_1=2$. \\ \\ 
    The Heronian minor relation we get is 
    \begin{equation*}
        -m_{1,n,2}m_{1,j_2,j_2+1}+m_{1,n,j_2}m_{1,2,j_2+1}-m_{1,n,j_2+1}m_{1,2,j_2}=0,
    \end{equation*}
    and the corresponding $S$-relation is 
    \begin{equation*}
         S_{n,1,2}S_{1,j_2,j_2+1}-S_{n,1,j_2}S_{1,2,j_2+1}+S_{n,1,j_2+1}S_{1,2,j_2}=0,
    \end{equation*}
    where $2<j_2<n$. \\ \\ or
    \item $1\leq i_1 \leq n-1, i_2=i_1+1, j_3=j_2+1, j_1=j_0+1$. \\ 
    Since $ \{i_1,i_2\} \cap \{j_0,j_1,j_2,j_3\}=\{j_0\}$, we conclude that $j_0=i_2=i_1+1$. \\ \\
    The Heronian minor relation we get is 
    \begin{equation*}
        -m_{i_1,i_1+1,i_1+2}m_{i_1+1,j_2,j_2+1}+m_{i_1,i_1+1,j_2}m_{i_1+1,i_1+2,j_2+1}-m_{i_1,i_1+1,j_2+1}m_{i_1+1,i_1+2,j_2}=0,
    \end{equation*}
    and the corresponding $S$-relation is 
    \begin{equation*}
        -S_{i_1,i_1+1,i_1+2}S_{i_1+1,j_2,j_2+1}+S_{i_1,i_1+1,j_2}S_{i_1+1,i_1+2,j_2+1}-S_{i_1,i_1+1,j_2+1}S_{i_1+1,i_1+2,j_2}=0, 
    \end{equation*}
    where $i_1+2<j_2<n$. 
\end{itemize}
This completes the proof of part (b). \\ \\ Proof of (c).
Assume now that $\{i_1,i_2\} \cap \{j_0,j_1,j_2,j_3\} = \{j_1\}$. Then $m_{i_1i_2j_1}=0$, so  \eqref{eqn_plucker} becomes
\begin{equation}
   m_{i_1i_2j_0}m_{j_1j_2j_3}+m_{i_1i_2j_2}m_{j_0j_1j_3}-m_{i_1i_2j_3}m_{j_0j_1j_2}=0.
   \label{nznm}
\end{equation} 
 As earlier in the proof, we conclude that \eqref{nznm} is a Heronian minor relation if and only if
\begin{equation*}
    I \land [(J_{12}\lor J_{23}) \land (J_{01}  \lor J_{30}) \land (J_{01} \lor J_{12})].
\end{equation*}

\begin{table}[h!]
 \centering
 \begin{tabular}{|c|c|c|c|c|c|c|c|c|c|c|}
 \hline
 $\mathbf{J_{1,2} \lor J_{2,3}}$ & $\mathbf{J_{0,1} \lor J_{3,0}}$ &  $\mathbf{J_{0,1} \lor J_{1,2}}$  & \textbf{conjuction }\\
 \hline 
   $J_{1,2}$ & $J_{0,1}$  & $J_{0,1}$ & $J_{0,1}\land J_{1,2}$ 
 \\ \hline
 $J_{1,2}$ & $J_{0,1}$ &  $J_{1,2}$ & $J_{0,1}\land J_{1,2} $ 
 \\ \hline
 $J_{1,2}$ & $J_{3,0}$  & $J_{0,1}$ & $J_{0,1} \land J_{1,2} \land J_{3,0} $ 
 \\ \hline
  $J_{1,2}$ & $J_{3,0}$ &  $J_{1,2}$ & $J_{1,2} \land J_{3,0} $  \\
 \hline
 $J_{2,3}$ & $J_{0,1}$ &  $J_{0,1}$ & $J_{0,1} \land J_{2,3} $  \\
 \hline
 $J_{2,3}$ & $J_{0,1}$  & $J_{1,2}$ &   $J_{0,1} \land J_{1,2} \land J_{2,3} $ \\
 \hline
 $J_{2,3}$ & $J_{3,0}$ &  $J_{0,1}$ & $J_{0,1} \land  J_{2,3} \land J_{3,0}$ \\ 
 \hline
 $J_{2,3}$ & $J_{3,0}$ &  $J_{1,2}$ & $J_{1,2} \land  J_{2,3} \land J_{3,0}$ \\ 
 \hline
 
 \end{tabular}
 \caption{}
 \label{table3}
 \end{table} 
 Using Table \ref{table3} and the same reasoning as in the previous cases, we get that, in case $ \{i_1,i_2\} \cap \{j_0,j_1,j_2,j_3\}=\{j_1\}$, \eqref{nznm} is a Heronian minor relation if and only if 
 \begin{multline*}
 I\land [ (J_{0,1} \land J_{1,2}) \lor (J_{0,1} \land J_{1,2} \land J_{3,0}) \lor (J_{1,2} \land J_{3,0}) \lor (J_{0,1} \land J_{2,3}) \\ \lor (J_{0,1} \land J_{1,2} \land J_{2,3}) \lor (J_{0,1} \land J_{2,3} \land J_{3,0}) \lor (J_{1,2} \land J_{2,3} \land J_{3,0})],
 \end{multline*}
which is equivalent to 
\begin{equation*}
   (I\land J_{0,1} \land J_{1,2}) \lor (I \land J_{1,2} \land J_{3,0}) \lor (I \land J_{0,1} \land J_{2,3}).
\end{equation*}
 If $I \land J_{0,1} \land J_{1,2}$, one can easily check that both 
 \begin{itemize}
     \item 
 $i_1=1 \land i_2=n \land J_{0,1} \land J_{1,2}$ , and 
 \item $1\leq i_1 \leq n-1 \land i_2=i_1+1 \land J_{0,1} \land J_{1,2}$
 \end{itemize} 
 lead to contradiction, since $\{i_1,i_2\} \cap \{j_0,j_1,j_2,j_3\} = \{j_1\}$.
 \\ \\ If $I \land J_{1,2} \land J_{3,0}$, then again, since $\{i_1,i_2\} \cap \{j_0,j_1,j_2,j_3\}=\{j_1\}$, we conclude that 
 \begin{equation*}
     1 < i_1 < n-1, i_2=i_1+1 \text{ \quad and \quad} (j_0,j_1,j_2,j_3) = (1,i_1+1,i_1+2,n).
 \end{equation*}
 After using that $m_{1,i_1+1,n}=m_{n,1,i_1+1}$, the Heronian minor relation we get is 
 \begin{equation*}
     m_{i_1,i_1+1,1}m_{i_1+1,i_1+2,n}+m_{i_1,i_1+1,i_1+2}m_{n,1,i_1+1}-m_{i_1,i_1+1,n}m_{1,i_1+1,i_1+2}=0,
 \end{equation*}
 and the corresponding $S$-relation is 
 \begin{equation*}
   S_{i_1,i_1+1,1}S_{i_1+1,i_1+2,n}+S_{i_1,i_1+1,i_1+2}S_{n,1,i_1+1}-S_{i_1,i_1+1,n}S_{1,i_1+1,i_1+2}=0,  
 \end{equation*}
 where $1<i_1<n-1$. \\  \\ 
In case $I \land J_{0,1} \land J_{2,3}$, by examining all the possible cases, we get to the conclusion that the only one that can hold is 
\begin{equation*}
    1< i_1 \leq n-1, i_2=i_1+1, j_0=i_1-1, j_1=i_1, j_3=j_2+1.
\end{equation*}
The Heronian minor relation we get is 
\begin{equation*}
    m_{i_1,i_1+1,i_1-1}m_{i_1,j_2,j_2+1}+m_{i_1,i_1+1,j_2}m_{i_1-1,i_1,j_2+1}-m_{i_1,i_1+1,j_2+1}m_{i_1-1,i_1,j_2}=0,
\end{equation*}
and the corresponding $S$-relation is 
\begin{equation*}
    S_{i_1,i_1+1,i_1-1}S_{i_1,j_2,j_2+1}+S_{i_1,i_1+1,j_2}S_{i_1-1,i_1,j_2+1}-S_{i_1,i_1+1,j_2+1}S_{i_1-1,i_1,j_2}=0,
\end{equation*}
where $1<i_1<j_2$.
\\ This completes proof of part (c), and one similarly proves parts (d) and (e).
\end{proof}
\begin{remark}
It is easy to check that,  when \\ $(i_1,i_2) \in \{(j_0,j_1),(j_1,j_2),(j_2,j_3), (j_0,j_3)\}$, the corresponding Heronian minor relation is trivially true.
\end{remark}
\begin{example} Let $P=(A_1,A_2,...,A_6)$, and let $F$ be the corresponding Heronian frieze. Then \begin{equation*}
m_{561}m_{234} - m_{562}m_{134} + m_{563}m_{124} - m_{564}m_{123} = 0
\end{equation*} is a Heronian minor relation, and the corresponding $S$-relation is : 
\begin{equation*}
  S_{561}S_{234} - S_{562}S_{134} + S_{563}S_{124} - S_{564}S_{123} = 0,   
\end{equation*}while
\begin{equation*}
m_{241}m_{256} - m_{242}m_{156} + m_{245}m_{126} - m_{246}m_{125} = 0  
\end{equation*}
\\
is not a Heronian minor relation. 
\end{example}
\begin{example} Let $P=(A_1,A_2,...,A_6)$, and let $F$ be a corresponding polygonal Heronian frieze. The list of all Heronian minor relations is the following:
\begin{align*}
    m_{162}m_{345} -m_{163}m_{245}+m_{164}m_{235}-m_{165}m_{234}&= 0; \\ 
    m_{123}m_{456}-m_{124}m_{356}+m_{125}m_{346}-m_{126}m_{345}&=0;\\
    m_{341}m_{256}-m_{342}m_{156}+m_{345}m_{126}-m_{346}m_{125}&=0; \\ 
    m_{561}m_{234}-m_{562}m_{134}+m_{563}m_{124}-m_{564}m_{123}&=0;\\ 
    m_{231}m_{456}-m_{234}m_{156}+m_{235}m_{146}-m_{236}m_{145}&=0; \\
    m_{451}m_{236}-m_{452}m_{136}+m_{453}m_{126}-m_{456}m_{123}&=0;\\
    -m_{123}m_{146}+m_{124}m_{136}-m_{126}m_{134}&=0;\\
    -m_{124}m_{156}+m_{125}m_{146}-m_{126}m_{145}&=0;\\
    -m_{123}m_{245}+m_{124}m_{235}-m_{125}m_{234}&=0;\\
    -m_{123}m_{256}+m_{125}m_{236}-m_{126}m_{235}&=0;\\
    -m_{234}m_{356}+m_{235}m_{346}-m_{236}m_{345}&=0;\\
    m_{231}m_{346}+m_{234}m_{136}-m_{236}m_{134}&=0;\\
    m_{341}m_{456}+m_{345}m_{146}-m_{346}m_{145}&=0;\\ 
    m_{341}m_{245}-m_{342}m_{145}-m_{345}m_{124}&=0;\\
    m_{451}m_{256}-m_{452}m_{156}-m_{456}m_{125}&=0;\\
    m_{452}m_{356}-m_{453}m_{256}-m_{456}m_{235}&=0;\\
    m_{561}m_{236}-m_{562}m_{136}+m_{563}m_{126}&=0;\\
    m_{561}m_{346}-m_{563}m_{146}+m_{564}m_{136}&=0.\\
\end{align*}
 \label{primjer}
\end{example}
Having a Heronian minor relation, it may also be of interest to give a graphical interpretation of the corresponding $S$-relation, in terms of the positions in the frieze.
Namely, one can notice that the choice of tuples $1 \leq i_1 < i_2 \leq n$ and $1 \leq j_0 < j_1 < j_2 <j_3 \leq n$ corresponds to the choice of a dashed line and a diamond of the frieze, respectively. Hence the following corollary.
\begin{corollary}
Each Heronian minor relation gives a relation between the $S$-entries of a diagonal (dashed line) and $S$-entries of a diamond of the Heronian frieze. \\ Namely, tuples $1 \leq i_1 < i_2 \leq n$ and $1 \leq j_0 < j_1 < j_2 <j_3 \leq n$, provided that they satisfy conditions of Theorem \ref{main}, give a relation between the entries $S_{i_1i_2j_0}, S_{i_1i_2j_1}, S_{i_1i_2j_2}, S_{i_1i_2j_3}$ of the diagonal $x_{i_1i_2}$ and the entries $S_{j_1j_2j_3}, S_{j_0j_2j_3}, \\ S_{j_0j_1j_3}, S_{j_0j_1j_2}$ of the diamond $j_0j_1j_2j_3$. This is due to the fact that the $S$-entries 'lying' on the dashed lines of a Heronian frieze are of the form \\ $S_{a,a+1,a+1}, S_{a,a+1,a+2},...,S_{a,a+1,n},...,S_{a,a+1,a}$, going from the south-west to the north-east, respectively, where $a \in \{1,2.,...,n\}$, and index addition is modulo $n$, as well as the fact that all the diamonds of the Heronian frieze of order $n$ correspond to the quadruples of vertices of form $(A_b,A_{b+1},A_c,A_{c+1})$, where $b,c$ are distinct elements of $\{1,2,...,n\}$ and addition is modulo $n$, as stated in Remark 1.10 in \cite{key2}. \\ The actual relation is either \eqref{srln} or \eqref{srln1}, depending on the choice of $(i_1,i_2)$.
The corresponding positions in the frieze are presented in  Figure \ref{figura} and Figure \ref{figuraaa}, depending on the tuple $(j_0,j_1,j_2,j_3)$ being equal to $(j_0,j_0+1,j_2,j_2+1)$, or $(1,j_1,j_1+1,n)$, respectively, where the same coloured $S$-entries are part of the same summand of the relation. 
\end{corollary}
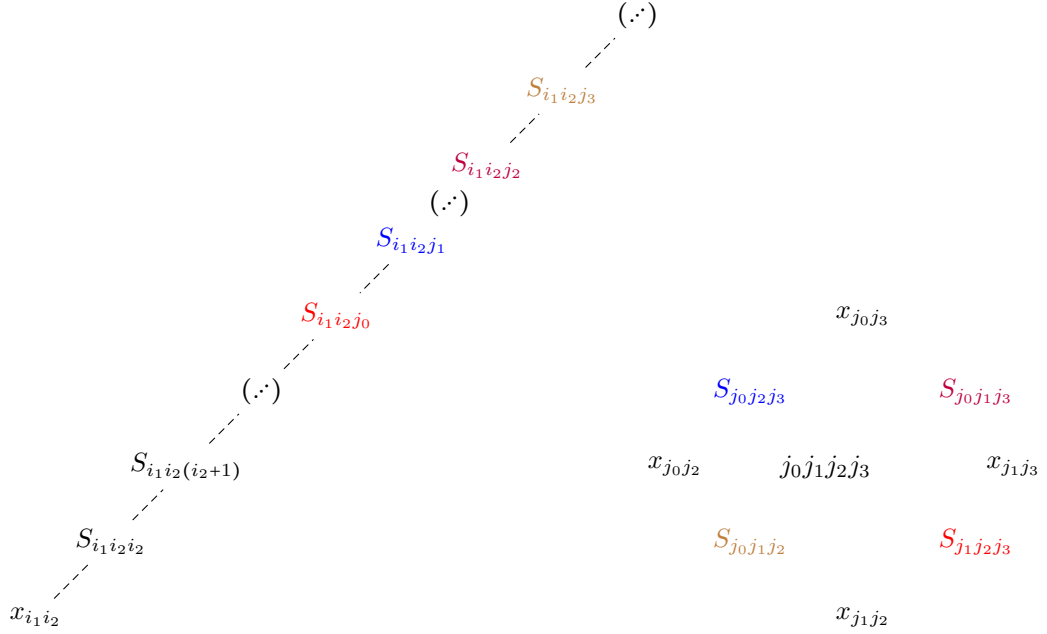
\begin{figure}
\begin{center}
\begin{tikzpicture}
  \node[draw=none](n24252) at (-0.5,0) {$x_{i_1i_2}$};    
  \node[draw=none](n2) at (0.5,1) {$S_{i_1i_2i_2}$};
   \node(n11) at (1.5,2) {$S_{i_1i_2(i_2+1)}$};
  \node(n17) at (2.5,3) {$(\udots)$}; 
  \node[red,very thick](n271) at (3.5,4) {$S_{i_1i_2j_0}$};
  \draw[densely dashed](n17)--(n271);
   \node[blue,very thick](n27) at (4.5,5) {$S_{i_1i_2j_1}$};
   \node(n1995911) at (5,5.5) {$(\udots)$};
   \node[purple,very thick](n415) at (5.5,6){$S_{i_1i_2j_2}$};
   
   \node[brown, very thick](n420) at (6.5,7){$S_{i_1i_2j_3}$};
    \node(n421) at (7.5,8){$(\udots)$};
   \draw[densely dashed](n420)--(n421);
   \draw[densely dashed](n415)--(n420);
   \draw[densely dashed](n27)--(n271);
   \draw[densely dashed] (n24252) -- (n2);
\draw[densely dashed] (n2) -- (n11);
\draw[densely dashed] (n11) -- (n17);
\node[draw=none](n59) at (8,2) {$x_{j_0j_2}$};
\node[draw=none](n103) at (10.5,4) {$x_{j_0j_3}$};
\node[blue,very thick](n104) at (9,3) {$S_{j_0j_2j_3}$};
  \node[draw=none] (n141415151) at (10,2) {\textbf{$j_0j_1j_2j_3$}};
\node[brown,very thick](n73) at (9,1) {$S_{j_0j_1j_2}$};
\node[draw=none](n106) at (12.5,2) {$x_{j_1j_3}$};
\node[purple,very thick](n106) at (12,3) {$S_{j_0j_1j_3}$};
\node[draw=none](n74) at (10.5,0) {$x_{j_1j_2}$};
\node[red,very thick](n107) at (12,1) {$S_{j_1j_2j_3}$};
\end{tikzpicture}
\end{center}
\caption{Graphical interpretation of the $S$-relation from Theorem \ref{main}\\ in case $(j_0,j_1,j_2,j_3)=(j_0,j_0+1,j_2,j_2+1)$}
\label{figura}
\end{figure}
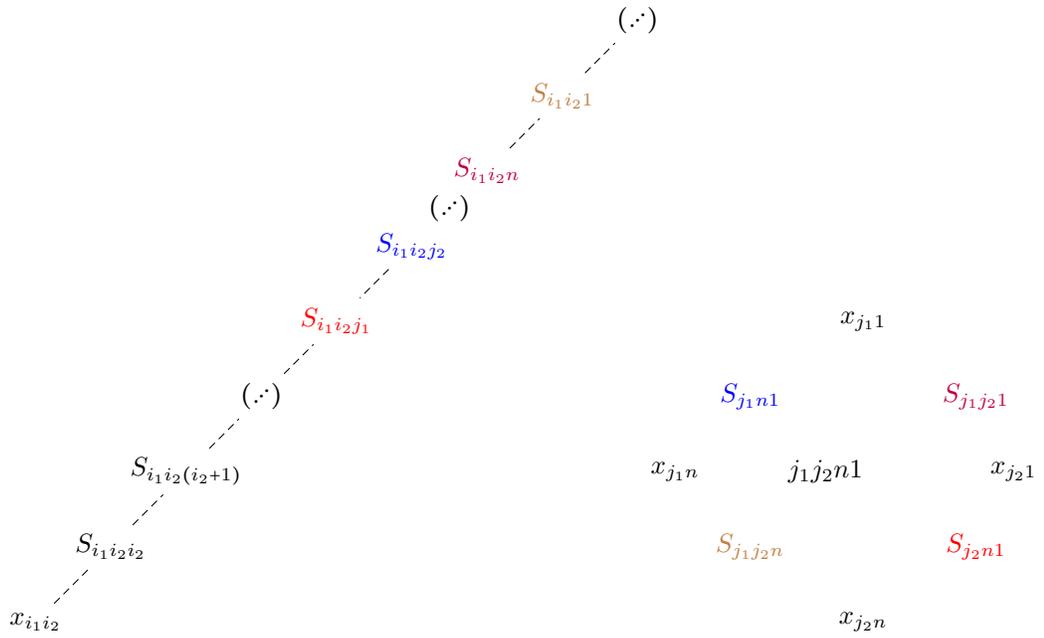
\begin{figure}
\begin{center}
\begin{tikzpicture}
  \node[draw=none](n24252) at (-0.5,0) {$x_{i_1i_2}$};    
  \node[draw=none](n2) at (0.5,1) {$S_{i_1i_2i_2}$};
   \node(n11) at (1.5,2) {$S_{i_1i_2(i_2+1)}$};
  \node(n17) at (2.5,3) {$(\udots)$}; 
  \node[red,very thick](n271) at (3.5,4) {$S_{i_1i_2j_1}$};
  \draw[densely dashed](n17)--(n271);
   \node[blue,very thick](n27) at (4.5,5) {$S_{i_1i_2j_2}$};
   \node(n1995911) at (5,5.5) {$(\udots)$};
   \node[purple,very thick](n415) at (5.5,6){$S_{i_1i_2n}$};
   
   \node[brown, very thick](n420) at (6.5,7){$S_{i_1i_21}$};
    \node(n421) at (7.5,8){$(\udots)$};
   \draw[densely dashed](n420)--(n421);
   \draw[densely dashed](n415)--(n420);
   \draw[densely dashed](n27)--(n271);
   \draw[densely dashed] (n24252) -- (n2);
\draw[densely dashed] (n2) -- (n11);
\draw[densely dashed] (n11) -- (n17);
\node[draw=none](n59) at (8,2) {$x_{j_1n}$};
\node[draw=none](n103) at (10.5,4) {$x_{j_11}$};
\node[blue,very thick](n104) at (9,3) {$S_{j_1n1}$};
  \node[draw=none] (n141415151) at (10,2) {\textbf{$j_1j_2n1$}};
\node[brown,very thick](n73) at (9,1) {$S_{j_1j_2n}$};
\node[draw=none](n106) at (12.5,2) {$x_{j_21}$};
\node[purple,very thick](n106) at (12,3) {$S_{j_1j_21}$};
\node[draw=none](n74) at (10.5,0) {$x_{j_2n}$};
\node[red,very thick](n107) at (12,1) {$S_{j_2n1}$};
\end{tikzpicture}
\end{center}
\caption{Graphical interpretation of the $S$-relation from Theorem \ref{main}\\ in case $(j_0,j_1,j_2,j_3)=(1,j_1,j_1+1,n)$}
\label{figuraaa}
\end{figure}
\begin{example}
\label{ex414}
Let F be a polygonal Heronian frieze of order 6. Then choosing $(i_1, i_2) = (2, 3)$ and $(j_0,j_1,j_2,j_3) = (3,4,5,6)$ corresponds to choosing the dashed line (diagonal) $x_{23}$ and the diamond $3456$, and the corresponding $S$-relation from Theorem \ref{main} is the following:
\begin{equation*}
S_{233}S_{456} - S_{234}S_{356} + S_{235}S_{346} - S_{236}S_{345} = 0.
 \end{equation*}
 Although $S_{233}S_{456}=0$, we keep it to illustrate the pattern.
A graphical interpretation is given in  Figure \ref{figura2}.  
\end{example}
Following on what has already been said about relation between $S$  entries of the polygonal Heronian frieze and minors of the corresponding  coordinate matrix, we use Lemma \ref{lem11} to establish a connection between some subfriezes of a Heronian frieze and Plücker friezes $P(3,n)$, as introduced in \cite{key6}. Firstly, we recall the definition of a Plücker frieze, as stated in \cite{key6}, as well as some of its properties. \\ \\ 
As stated in \cite{key6}, we fix $k, n \in \mathbb{Z}_{>0}$ with $2 \leq k \leq \frac{n}{2}$ and consider the Grassmannian $Gr(k,n)$ as a projective variety via the Plücker embedding, with homogenous coordinate ring 
\begin{equation*}
    {\mathit{A}}(k,n) = \mathbb{C}\left[Gr(k,n) \right].
\end{equation*}
To the homogenous coordinate ring ${\mathit{A}}(k,n)$  we now associate a certain frieze pattern. As stated in \cite{key6}, for brevity, it is convenient to write $[r]^l$ for the set $\{r,r+1,...,r+l-1\}$.
\begin{definition} \cite{key6}
The \textit{Plücker frieze} of type $(k,n)$ is a $\mathbb{Z} \times \{1,2,...,n+k-1\}$-grid with entries given by the map
\begin{equation*}
\varphi_{(k,n)}: \mathbb{Z} \times \{1,2,...,n+k-1\} \to A(k,n), \qquad (r,m) \mapsto p_{o([r']^{k-1},m')},
\end{equation*}
where $r' \in [1,n]$ is the reduction of $r$ modulo $n$ and $m' \in [1,n]$ is the reduction of $m+r'-1$ modulo $n$. \\ We denote the Plücker frieze by $P_{(k,n)}$, and it is presented in  Figure \ref{plfriz}.
\begin{figure}
\begin{center}
\includegraphics[scale=0.7]{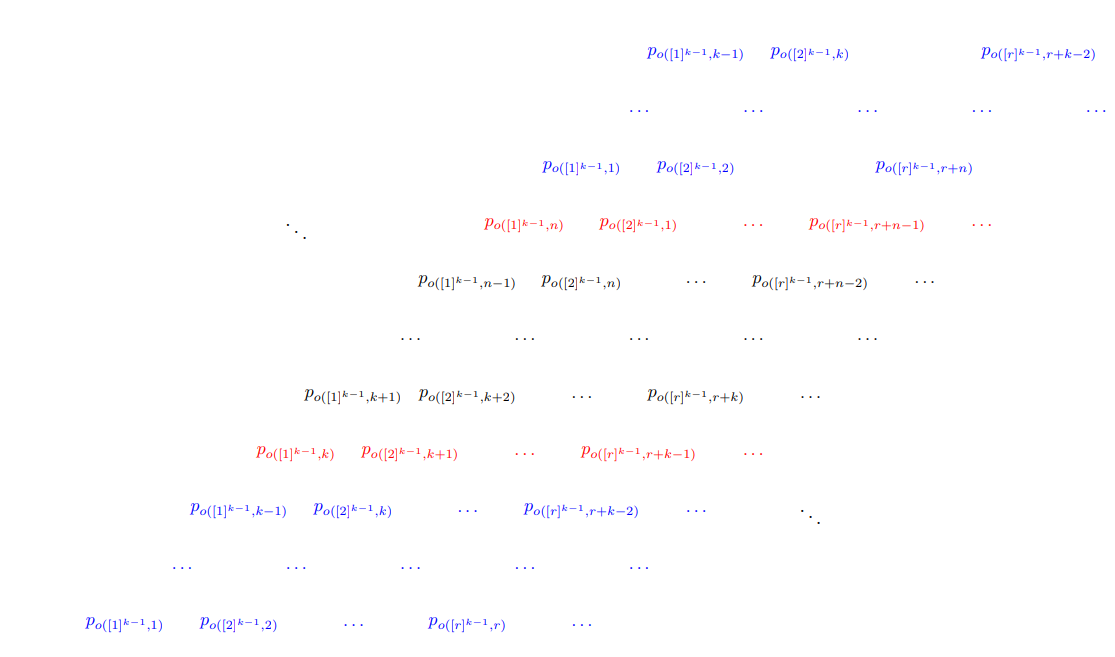}
\end{center}
\caption{Plücker frieze of type $(k,n)$, where indices are reduced modulo $n$. The top and bottom $k-1$ rows (blue) are $0$.}
\label{plfriz}
\end{figure}
\label{def311}
\end{definition}
\begin{example}
The Plücker frieze for $k=3$ is given in Figure \ref{figura6}.
\begin{figure}
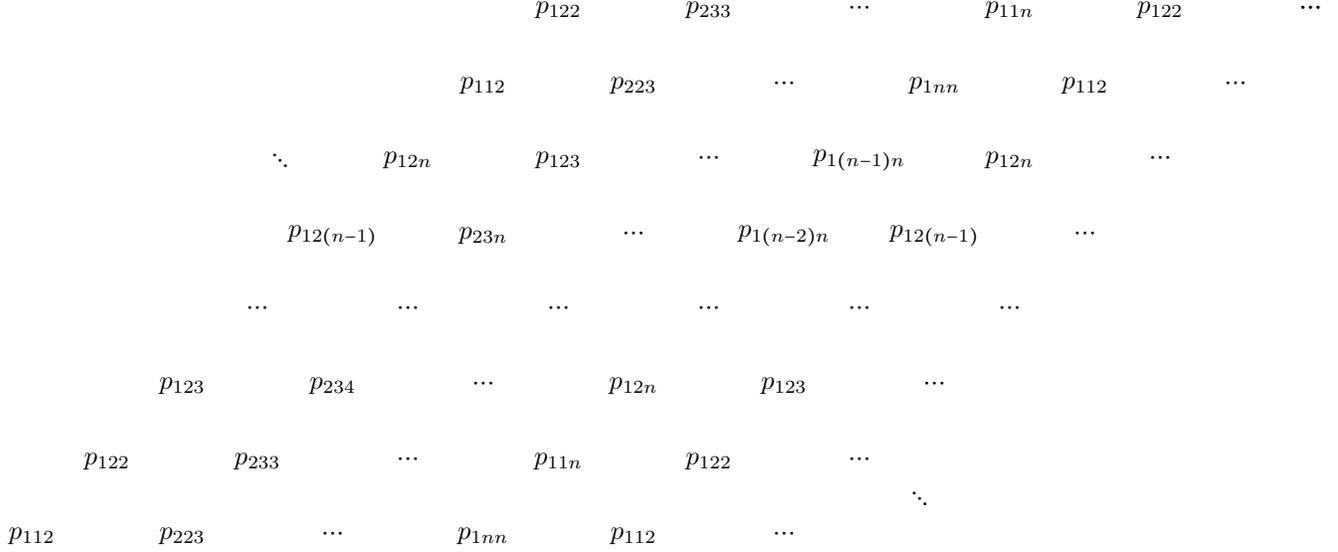

\hspace*{-1.3cm}

\caption{The Plücker frieze $P(3,n)$}
\label{figura6}
\end{figure}
\end{example}
\begin{example}
The Plücker frieze of type $(3,6)$ is given in Figure \ref{pluckerex}.
\begin{sidewaysfigure}
\hspace*{-3.5cm}
%
\caption{Fragment of Heronian frieze of order 8}
\label{fig,labelings}
\end{sidewaysfigure}
\end{example}

\begin{figure}
\hspace*{-1.5cm}
%
\caption{Heronian minor frieze of order $n$}
\label{min_subfr}
\end{figure}

\begin{figure}
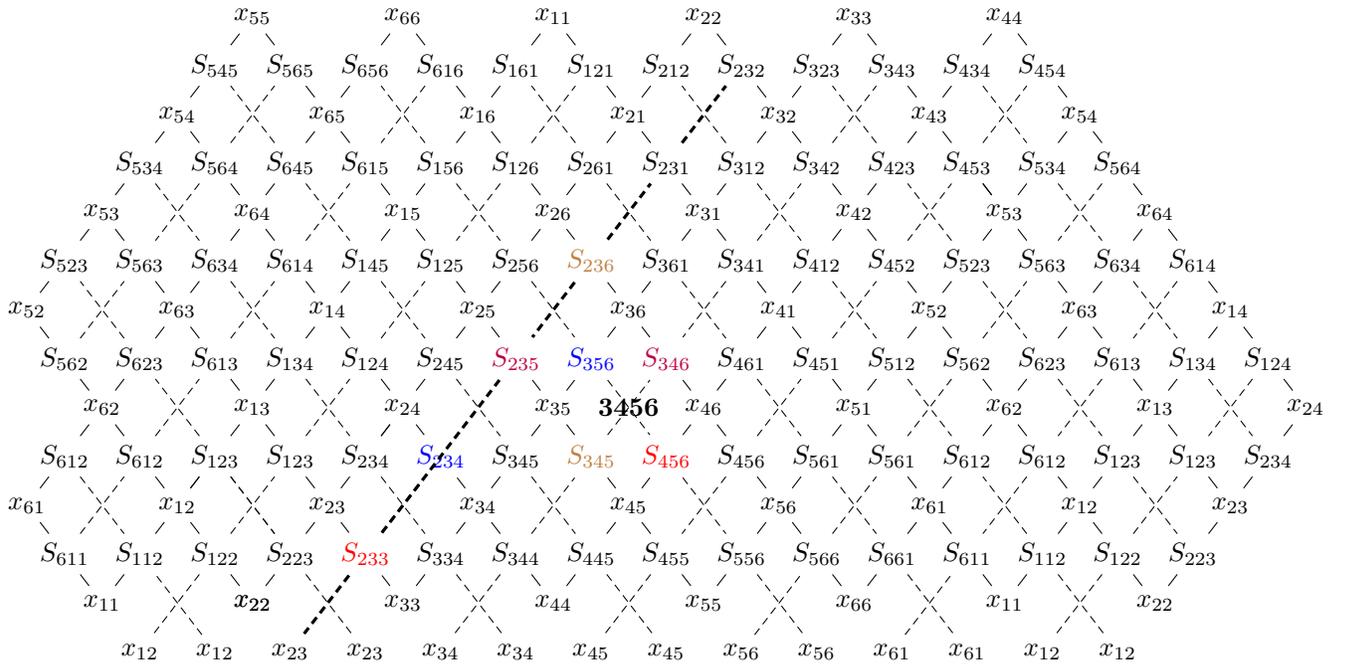

\hspace*{-2cm}
%
\caption{Graphical interpretation of the $S$-relation from Example \ref{ex414}}
\label{figura2}
\end{figure}

\begin{definition}\cite{key6} 
The $k \times k$ diamonds in the Plücker frieze $P_{(k,n)}$ are matrices of the form 
\begin{equation}  
A_{[m]^k;r} := (p_{o([r+i-1]^{k-1}, m+j-1)})_{1 \leq i, j \leq k},
\end{equation}
with $r \in [1,n]$ and $m \in [r+k-1, r+n-1] = [r+k-1, r-1]$, where calculation is modulo $n$, with representatives $1,...,n$.
\\ \\ In other words, as stated in \cite{key6}, a $k \times k$ diamond in the Plücker frieze $P_{(k,n)}$ is a matrix having $k$ successive entries of a row of the frieze $P_{(k,n)}$ on its diagonal and shorter rows above and below the diagonal accordingly.
\end{definition}
\noindent The proof of the following theorem can be found in \cite{key6}.
\begin{theorem}
All determinants of $(k+1) \times (k+1)$ diamonds of a Plücker frieze $P_{(k,n)}$ vanish.
\label{thmvan}
\end{theorem}
\begin{example}
The following holds for the entries of the Plücker frieze $P_{(3,6)}$:
    \[
\begin{vmatrix}   
p_{124} & p_{125} & p_{126} & 0 \\
p_{234} & p_{235} & p_{236} & p_{123}\\
0 & p_{345} & p_{346} & p_{134} \\
0 & 0 & p_{456} & p_{145}
\end{vmatrix} = 0,
\]
where the corresponding Plücker coordinates are shown in red in Figure \ref{pluckerex}.
\end{example}
\begin{definition}
 Let $P=(A_1,A_2,...,A_n)$, with $n \geq 3$, be a polygon in the complex plane, and let $C$ be the corresponding coordinate matrix. A \textit{Heronian minor frieze of order $n$} is a collection of minors of matrix $C$ arranged as in Figure \ref{min_subfr}.
\end{definition}
\begin{definition}
    A $4 \times 4$ diamond in a Heronian minor frieze is a matrix having four successive entries of a row of the frieze on its diagonal and shorter rows above and below the diagonal accordingly.
    \label{definicija}
\end{definition}
\begin{lemma}
    Determinants of $4 \times 4$ diamonds of a Heronian minor frieze vanish.
    \label{lem317}
\end{lemma}
\begin{proof}
This follows from Lemma \ref{lem11}, Theorem \ref{thmvan}, and the fact that $m_{ijk} = m_{kij} = m_{jki}$, for $i,j,k \in \{1,2,...,n\}$.
\end{proof}
\begin{definition}
\label{podfriz}
Let $P = (A_1,A_2,...,A_n)$, with $n \geq 3$ be a polygon in the complex plane, with its corresponding Heronian frieze $F$. An \textit{$S$-subfrieze} of $F$ is a collection of the $S$-entries of F on alternating diagonals, as presented in Figure \ref{subfr}.
\end{definition}
\begin{definition}
Let $P = (A_1,A_2,...,A_n)$, with $n \geq 3$ be a polygon in the complex plane, and $F$ its corresponding Heronian frieze. 
 We define a $4 \times 4$ diamond of the $S$-subfrieze of $F$ analogously as we did for a Heronian minor frieze in Definition \ref{definicija}.
\end{definition}

\begin{lemma}
Let $F$ be a polygonal Heronian frieze of order $n$. Determinants of $4 \times 4$ diamonds of the $S$-subfrieze of $F$ are equal to 16 times the corresponding determinant of the Heronian minor frieze. 
\label{lem320}
\end{lemma}
\begin{proof}
Follows from Lemma \ref{lem1} and the formula for multiplying a determinant with a scalar.
\end{proof}
\begin{theorem}
\label{bitna}
   Let F be a polygonal Heronian frieze of order $n$. Then determinants of $4 \times 4$ diamonds of its $S$-subfrieze vanish.
\end{theorem}
\begin{proof}
Follows from Lemma \ref{lem317} and Lemma \ref{lem320}.
\end{proof}
\begin{example}
 Let $P=(A_1,A_2,...,A_8)$ be a labeled $8$-gon. Then we have that the following holds for the entries of the $S$-subfrieze of the corresponding polygonal Heronian frieze: 
 \[
\begin{vmatrix}   
S_{125} & S_{126} & S_{127} & S_{128} \\
S_{235} & S_{236} & S_{237} & S_{238}\\
S_{345} & S_{346} & S_{347} & S_{348} \\
S_{455}  & S_{456} & S_{457} & S_{458}
\end{vmatrix} = 0.
\]  
The corresponding picture is given in Figure \ref{fig,labelings}, where the entries highlighted
in red are the ones appearing in the determinant.
\end{example}
\begin{figure}
\begin{center}
  \hspace*{-2cm}  
\begin{tikzpicture}

     \node[draw=none](n420) at (-1.5,-0.5){$S_{n11}$};
     
     \node[draw=none](n91951) at (-3.5,-0.5){$\cdots$};
     \node[draw=none](n421) at (-0.5,0.5){$S_{n12}$};
     \draw[densely dashed](n420)--(n421);
     \node[draw=none](n422) at (0.5,1.5){$S_{n13}$};
     \node[draw=none](n1991) at (-1.5,1.5){$\cdots$};
     \node[draw=none](n423) at (1.5,2.5){$S_{n14}$};
     \draw[densely dashed] (n422)--(n423);
     \node[draw=none](n424) at (2,3){$\udots$};
     \node[draw=none](n425) at ( 2.5,3.5) {$S_{n1(n-2)}$};
     \node[draw=none](n898) at (0.5,3.5){$\cdots$};
     \node[draw=none](n426) at (3.5,4.5){$S_{n1(n-1)}$};
     \node[draw=none](n427) at (4.5,5.5){$S_{n1n}$};
   
     \node[draw=none](n1000) at (2.5,5.5) {$\cdots$};

     \draw[densely dashed](n426)--(n427);
     \draw[densely dashed](n425)--(n426);
     \draw[densely dashed](n421)--(n422);
     \node[draw=none](n76) at (6.5,-0.5){$S_{n11}$};
     \node[draw=none](n15919) at (8.5,-0.5){$\cdots$};
     \node[draw=none](n77) at (7.5,0.5){$S_{n12}$};
     \node[draw=none](n78) at (8.5,1.5){$S_{n13}$};
     \node[draw=none](n19591) at (10.5,1.5){$\cdots$};
     \node[draw=none](n90) at (9.5,2.5){$S_{n14}$};
     \draw[densely dashed](n78)--(n90);
     \node[draw=none](n591) at (10,3){$\udots$};
     \node[draw=none](n592) at (10.5,3.5){$S_{n1(n-2)}$};
     \node[draw=none](n15191) at (12.5,3.5){$\cdots$};
     \node[draw=none](n593) at (11.5,4.5){$S_{n1(n-1)}$};
     \node[draw=none](n600) at (12.5,5.5) {$S_{n1n}$};
     \node[draw=none](n1051404) at (14.5,5.5){$\cdots$};
     \draw[densely dashed](n593)--(n600);
   
     \draw[densely dashed](n592)--(n593);
     \draw[densely dashed](n77)--(n78);
     \draw[densely dashed](n76)--(n77);

     \node[draw=none](n29) at (4.5,-0.5){$S_{344}$};
     \node[draw=none](n591996919) at (5.5,-0.5){$\cdots$};
     \node[draw=none](n30) at (5.5,0.5){$S_{345}$};
     \draw[densely dashed](n29)--(n30);
     \node[draw=none](n31) at (6.5,1.5) {$S_{346}$};
      \draw[densely dashed](n30)--(n31);
     \node[draw=none](n32) at (7.5,2.5){$S_{347}$};
     \draw[densely dashed](n31)--(n32);
     \node[draw=none](n60) at (8.8,2.5){$\cdots$};
     \node[draw=none](n40) at (8.5,3.5) {$S_{341}$};
     \node[draw=none](n41) at (9.5,4.5) {$S_{342}$};
     \node[draw=none](n42) at (10.5,5.5){$S_{343}$};
    \node[draw=none](n56919169) at (11.5,5.5){$\cdots$};
     \draw[densely dashed](n41)--(n42);

     \draw[densely dashed] (n40) -- (n41);

     \node[draw=none](n33) at (8,3) {$\udots$};

     \node[draw=none](n15) at (2.5,-0.5){$S_{233}$};
     
     \node[draw=none](n16) at (3.5,0.5){$S_{234}$};
     \draw[densely dashed](n15)--(n16);
     \node[draw=none](n17) at (4.5,1.5) {$S_{235}$};
     \draw[densely dashed] (n16)--(n17);
     \node[draw=none](n18) at (5.5,2.5){$S_{236}$};
     \draw[densely dashed](n17)--(n18);
     \node[draw=none](n19) at (6,3){$\udots$};
     \node[draw=none](n20) at (6.5,3.5) {$S_{23n}$};
     \node[draw=none](n21) at (7.5,4.5){$S_{231}$};
     \node[draw=none](n22) at (8.5,5.5) {$S_{232}$};

     \draw[densely dashed] (n21)--(n22);
     \draw[densely dashed](n20)--(n21);
     
     \node[draw=none](n2) at (0.5,-0.5) {$S_{122}$};
     
     \node[draw=none](n3) at (1.5,0.5){$S_{123}$};
     \draw[densely dashed](n2)--(n3);
     \node[draw=none](n4) at (2.5,1.5){$S_{124}$};
     \draw[densely dashed](n3)--(n4);
     \node[draw=none](n5) at (3.5,2.5){$S_{125}$};
     \draw[densely dashed](n4)--(n5);
     \node[draw=none](n6) at (4,3){$\udots$};
     \node[draw=none](n7) at (4.5,3.5){$S_{12(n-1)}$};
     \node[draw=none](n12) at (5.5,4.5){$S_{12n}$};
     \node[draw=none](n8) at (6.5,5.5){$S_{121}$};
     
     \draw[densely dashed](n7)--(n12);
     \draw[densely dashed](n8)--(n12);
\end{tikzpicture}
\end{center}
\caption{$S$-subfrieze of a  Heronian frieze of order $n$}
\label{subfr}
\end{figure}
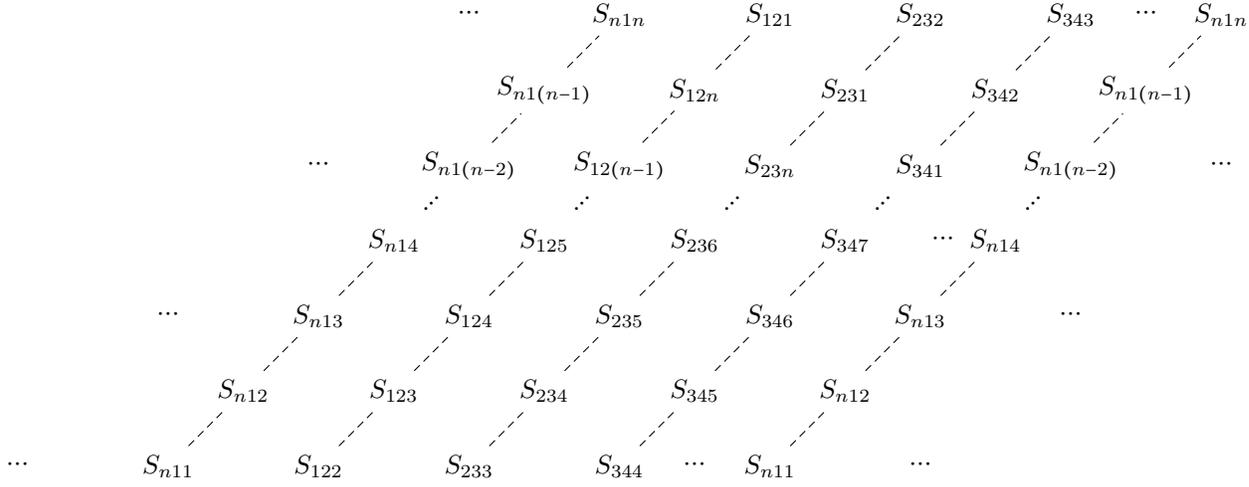
\begin{remark} 
\label{napomenaaa}
By taking $S$-entries of a Heronian frieze $F$, again on alternating diagonals, but now on the ones not chosen in Definition \ref{podfriz}, we get a different subfrieze consisting of $S$-entries. Similar methods to the ones from Theorem \ref{bitna} can be used to show that $4 \times 4$ diamonds of this subfrieze have vanishing determinants as well. An example of such a subfrieze is shown in Figure \ref{subfrrr}.
\end{remark}
\begin{figure}
\begin{center}
\begin{tikzpicture}[xscale=1,yscale=1.3]

 \node[draw=none](n3) at (1.5,2.5) {$0$};

 \node[draw=none](n12) at (2,3) {$0$};

  \node[draw=none](n1387445) at (6,5) {$p_{345}$};
  \node[draw=none](n51991) at (7,5) {$p_{456}$};
  \node[draw=none](n59151) at (8,5){$p_{156}$};
  \node[draw=none](n19519012) at (8.5,5.5) {$0$};
  \node[draw=none](n9419419) at (9,6) {$0$};
  \node[draw=none](n68484123) at (6.5,4.5) {$p_{356}$};
  
 \node[draw=none] (n15116) at (7,4) {$p_{136}$};
 \node[draw=none](n419149) at (8,4) {$p_{124}$};
 \node[draw=none](n919519196) at (7.5,3.5){$p_{123}$};
 \node[draw=none](n591951951) at (8.5,3.5){$p_{234}$};
 \node[draw=none](n411951951) at (8,3){$0$};
 \node[draw=none](n59195019591) at (8.5,2.5){$0$};
 \node[draw=none](n591519159) at (7.5,2.5){$0$};
 \node[draw=none](n41991851) at (7,3) {$0$};
 \node[draw=none](n4915919) at (6.5,2.5){$0$};
 \node[draw=none](n9419519) at (7.5,4.5) {$p_{146}$};
 \node[draw=none](n49199491) at (8.5,4.5){$\cdots$};
  \node[draw=none](n5101961) at (6.5,5.5){$0$};
  \node[draw=none](n9195195) at (7,6) {$0$};
  \node[draw=none](n591961) at (7.5,5.5) {$0$};
 \node[draw=none](n195919) at (8,6){$0$};
     \node[draw=none] (n15619690119) at (1.5,5.5){$0$};
    \node[draw=none](n591969119) at (2,6){$0$};
  \node[draw=none](n19087337) at (1.5,4.5) {$p_{146}$};

 \node[red,very thick](n18) at (2.5,3.5) {$p_{234}$};
 
\node[draw=none](n131) at (2,5) {$p_{156}$}; 
 \node[red,very thick](n21) at (2.5,4.5) {$p_{125}$};
\node[red,very thick](n515191) at (2,4){$p_{124}$};
\node[draw=none](n10) at (1.5,3.5) {$p_{123}$};
 \node[draw=none](n1) at (1,3) {$0$};
  \node[draw=none](n7) at (0.5,2.5){$0$};
 \node[draw=none](n4191994) at (1,4){$\cdots$};
 \node[red,very thick](n23) at (3,4) {$p_{235}$};

 \node[red,very thick](n266474) at (3.5,5.5) {$0$};
 \node[red,very thick](n08529251) at (3,5){$p_{126}$};
 \node[draw=none](n5919519) at (4,6){$0$};
\node[draw=none](n519519) at (5,6) {$0$};
 \node[draw=none](n151664335351) at (2.5,5.5){$0$};
\node[draw=none](n5195159) at (3,6) {$0$};

 \node[red,very thick](n31) at (3,3) {$0$};

\node[draw=none](n33) at (2.5,2.5) {$0$};

 \node[red,very thick](n38) at (3.5,3.5) {$p_{345}$};

  \node[red,very thick](n41) at (3.5,2.5) {$0$};

  \node[red,very thick](n45) at (4,4) {$p_{346}$};

\node[red,very thick](n48) at (4,3) {$0$};

\node[red,very thick](n51) at (4,5) {$p_{123}$};
\node[red,very thick](n5149191) at (3.5,4.5) {$p_{236}$};
\node[draw=none](n513452) at (4.5,5.5) {$0$};
\node[draw=none](n91519) at (6,6) {$0$};

\node[draw=none](n552929062) at (5.5,5.5){$0$};

\node[red,very thick](n53) at (4.5,4.5) {$p_{134}$};
\node[draw=none](n100) at (5,5) {$p_{234}$};

\node[draw=none](n103) at (5.5,4.5) {$p_{245}$};

\node[draw=none](n106) at (6,4) {$p_{256}$};

\node[draw=none](n109) at (6.5,3.5) {$p_{126}$};

\node[draw=none](n111) at (6,3) {$0$};

 \node[red,very thick](n56) at (4.5,3.5) {$p_{456}$};

 \node[red,very thick](n59) at (5,4) {$p_{145}$};

 \node[draw=none](n62) at (5,3) {$0$};

\node[draw=none](n64) at (4.5,2.5) {$0$};

\node[draw=none](n70) at (5.5,2.5) {$0$};

\node[draw=none](n74) at (5.5,3.5) {$p_{156}$};

\end{tikzpicture}
\end{center}
\caption{Fragment of Plücker frieze $P_{(3,6)}$}
\label{pluckerex}
\end{figure}
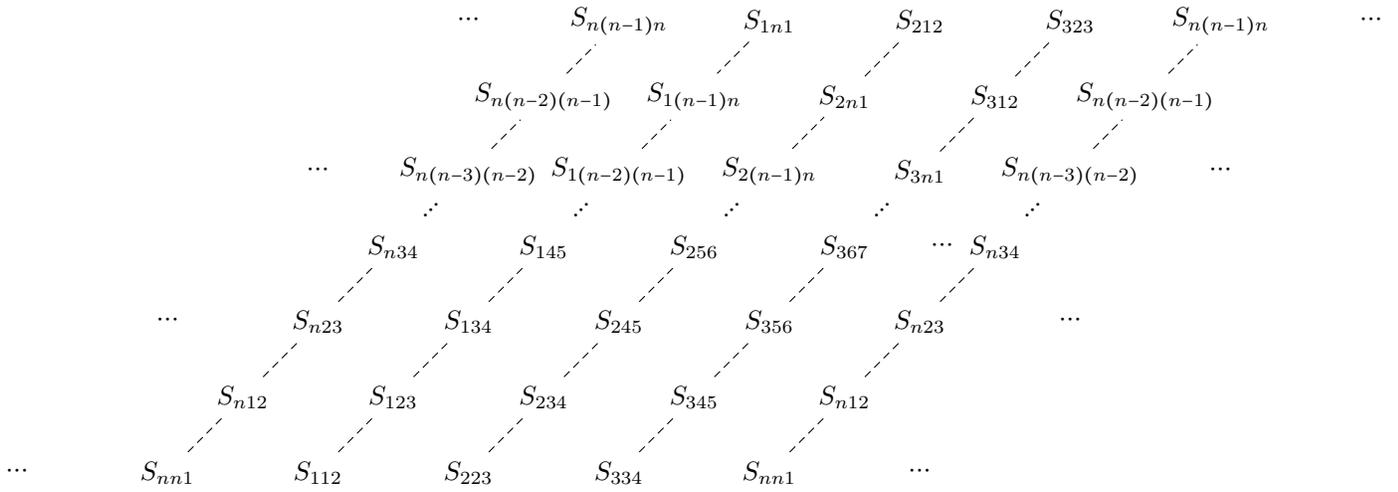
\begin{figure}
\begin{center}
  \hspace*{-2.5cm}  
\begin{tikzpicture}

     \node[draw=none](n420) at (-1.5,-0.5){$S_{nn1}$};
     
     \node[draw=none](n91951) at (-3.5,-0.5){$\cdots$};
     \node[draw=none](n421) at (-0.5,0.5){$S_{n12}$};
     \draw[densely dashed](n420)--(n421);
     \node[draw=none](n422) at (0.5,1.5){$S_{n23}$};
     \node[draw=none](n1991) at (-1.5,1.5){$\cdots$};
     \node[draw=none](n423) at (1.5,2.5){$S_{n34}$};
     \draw[densely dashed] (n422)--(n423);
     \node[draw=none](n424) at (2,3){$\udots$};
     \node[draw=none](n425) at ( 2.5,3.5) {$S_{n(n-3)(n-2)}$};
     \node[draw=none](n898) at (0.5,3.5){$\cdots$};
     \node[draw=none](n426) at (3.5,4.5){$S_{n(n-2)(n-1)}$};
     \node[draw=none](n427) at (4.5,5.5){$S_{n(n-1)n}$};
   
     \node[draw=none](n1000) at (2.5,5.5) {$\cdots$};

     \draw[densely dashed](n426)--(n427);
     \draw[densely dashed](n425)--(n426);
     \draw[densely dashed](n421)--(n422);
     \node[draw=none](n76) at (6.5,-0.5){$S_{nn1}$};
     \node[draw=none](n15919) at (8.5,-0.5){$\cdots$};
     \node[draw=none](n77) at (7.5,0.5){$S_{n12}$};
     \node[draw=none](n78) at (8.5,1.5){$S_{n23}$};
     \node[draw=none](n19591) at (10.5,1.5){$\cdots$};
     \node[draw=none](n90) at (9.5,2.5){$S_{n34}$};
     \draw[densely dashed](n78)--(n90);
     \node[draw=none](n591) at (10,3){$\udots$};
     \node[draw=none](n592) at (10.5,3.5){$S_{n(n-3)(n-2)}$};
     \node[draw=none](n15191) at (12.5,3.5){$\cdots$};
     \node[draw=none](n593) at (11.5,4.5){$S_{n(n-2)(n-1)}$};
     \node[draw=none](n600) at (12.5,5.5) {$S_{n(n-1)n}$};
     \node[draw=none](n1051404) at (14.5,5.5){$\cdots$};
     \draw[densely dashed](n593)--(n600);
   
     \draw[densely dashed](n592)--(n593);
     \draw[densely dashed](n77)--(n78);
     \draw[densely dashed](n76)--(n77);

     \node[draw=none](n29) at (4.5,-0.5){$S_{334}$};
     
     \node[draw=none](n30) at (5.5,0.5){$S_{345}$};
     \draw[densely dashed](n29)--(n30);
     \node[draw=none](n31) at (6.5,1.5) {$S_{356}$};
      \draw[densely dashed](n30)--(n31);
     \node[draw=none](n32) at (7.5,2.5){$S_{367}$};
     \draw[densely dashed](n31)--(n32);
     \node[draw=none](n60) at (8.8,2.5){$\cdots$};
     \node[draw=none](n40) at (8.5,3.5) {$S_{3n1}$};
     \node[draw=none](n41) at (9.5,4.5) {$S_{312}$};
     \node[draw=none](n42) at (10.5,5.5){$S_{323}$};
    
     \draw[densely dashed](n41)--(n42);

     \draw[densely dashed] (n40) -- (n41);

     \node[draw=none](n33) at (8,3) {$\udots$};

     \node[draw=none](n15) at (2.5,-0.5){$S_{223}$};
     
     \node[draw=none](n16) at (3.5,0.5){$S_{234}$};
     \draw[densely dashed](n15)--(n16);
     \node[draw=none](n17) at (4.5,1.5) {$S_{245}$};
     \draw[densely dashed] (n16)--(n17);
     \node[draw=none](n18) at (5.5,2.5){$S_{256}$};
     \draw[densely dashed](n17)--(n18);
     \node[draw=none](n19) at (6,3){$\udots$};
     \node[draw=none](n20) at (6.5,3.5) {$S_{2(n-1)n}$};
     \node[draw=none](n21) at (7.5,4.5){$S_{2n1}$};
     \node[draw=none](n22) at (8.5,5.5) {$S_{212}$};

     \draw[densely dashed] (n21)--(n22);
     \draw[densely dashed](n20)--(n21);
     
     \node[draw=none](n2) at (0.5,-0.5) {$S_{112}$};
     
     \node[draw=none](n3) at (1.5,0.5){$S_{123}$};
     \draw[densely dashed](n2)--(n3);
     \node[draw=none](n4) at (2.5,1.5){$S_{134}$};
     \draw[densely dashed](n3)--(n4);
     \node[draw=none](n5) at (3.5,2.5){$S_{145}$};
     \draw[densely dashed](n4)--(n5);
     \node[draw=none](n6) at (4,3){$\udots$};
     \node[draw=none](n7) at (4.5,3.5){$S_{1(n-2)(n-1)}$};
     \node[draw=none](n12) at (5.5,4.5){$S_{1(n-1)n}$};
     \node[draw=none](n8) at (6.5,5.5){$S_{1n1}$};
     
     \draw[densely dashed](n7)--(n12);
     \draw[densely dashed](n8)--(n12);
\end{tikzpicture}
\end{center}
\caption{Subfrieze from Remark \ref{napomenaaa}}
\label{subfrrr}
\end{figure}

\end{document}